\documentclass[11pt]{amsart}
\pdfoutput=1

\usepackage[dvipsnames, table, xcdraw]{xcolor}
\usepackage{upgreek}
\usepackage{amssymb,amsmath,mathtools,amsfonts,amsthm}
\usepackage[mathscr]{eucal}
\usepackage{fullpage}
\usepackage{color}
\usepackage{caption}
\usepackage{subcaption}
\usepackage{enumitem}
\usepackage{soul}

\usepackage[colorlinks]{hyperref}
\definecolor{goodcolor}{RGB}{240,255,240}
\definecolor{headercolor}{RGB}{255,255,240}
\definecolor{mylinkcolor}{RGB}{0,0,255}
\definecolor{mycitecolor}{RGB}{169,169,169}
\definecolor{myurlcolor}{RGB}{255,20,147}
\hypersetup{
  colorlinks=true,
  urlcolor=myurlcolor,
  citecolor=mycitecolor,
  linkcolor=mylinkcolor,
  linktoc=page,
  breaklinks=true
}

\usepackage{url} 
\usepackage[numbers]{natbib}
\usepackage{setspace}  
\usepackage{indentfirst} 
\usepackage{listings} 

\usepackage{colonequals} 
\usepackage{nicefrac}

\usepackage{tikz-cd,tikz} 
\usepackage{cleveref} 
\usetikzlibrary{positioning}

\usepackage{float}
\usepackage{adjustbox} 
\newcommand{\PreserveBackslash}[1]{\let\temp=\\#1\let\\=\temp}
\newcolumntype{C}[1]{>{\PreserveBackslash\centering}p{#1}}
\usepackage{diagbox}
\usepackage{multirow}
\usepackage{hhline}
\usepackage{booktabs}
\usepackage{placeins}

\usepackage{algorithmic}
\usepackage[ruled]{algorithm}
\makeatletter 
 
\@addtoreset{algorithm}{section} 
\makeatother

\numberwithin{equation}{section}

\newtheorem{theorem}[algorithm]{Theorem}

\newtheorem{lemma}[algorithm]{Lemma}
\newtheorem{coro}[algorithm]{Corollary}
\newtheorem{conjecture}[algorithm]{Conjecture}
\newtheorem{proposition}[algorithm]{Proposition}

\theoremstyle{definition} 
\newtheorem{defn}[algorithm]{Definition}
 
\newtheorem{question}[algorithm]{Question}

\newtheorem{remark}[algorithm]{Remark}

\numberwithin{table}{section}
\numberwithin{figure}{section}
\newcommand{\oalpha}{\overline{\alpha}}

\newcommand{\ZZ}{\mathbf{Z}} 
\newcommand{\QQ}{\mathbf{Q}} 
\newcommand{\CC}{\mathbf{C}} 
\newcommand{\FF}{\mathbf{F}} 
 
\newcommand{\Fq}{\FF_q}

\newcommand{\Qbar}{\overbar{\QQ}}
\newcommand{\Fqbar}{\overbar{\FF}_q}

\newcommand{\OO}{\mathscr{O}}
\newcommand{\mfp}{\mathfrak{p}}
\newcommand{\dpr}{\mathfrak{A}}
\newcommand{\mfP}{\mathfrak{P}}

\newcommand{\brk}[1]{ \mathopen{}\left\lbrace #1 \right\rbrace\mathclose{}}

\newcommand{\cdef}[1]{{\color{blue}{\textit{#1}}}}

\newcommand{\ideal}[1]{\langle #1 \rangle}

\newcommand{\cO}{\mathcal{O}} 
\newcommand{\cP}{\mathcal{P}}
\newcommand{\cR}{\mathcal{R}}
\newcommand{\vecsp}[2]{%
\ensuremath%
#1\langle#2\rangle%
}
\newcommand{\agrp}[1]{\Gamma\!_{#1}}

\newcommand{\exceptional}[3]{\mathcal{E}_{#1}(#2,#3)}
\newcommand{\decor}[1]{\ensuremath \mathsf{#1}}

\renewcommand{\i}{\mathsf{i}}
\renewcommand{\j}{\mathsf{j}}

\newcommand{\g}{\mathsf{g}}
\newcommand{\1}{\mathsf{1}}

\newcommand{\overbar}[1]{\mkern 1.5mu\overline{\mkern-1.5mu#1\mkern-1.5mu}\mkern 1.5mu}

\DeclareMathOperator{\ord}{ord} 
 
\DeclareMathOperator{\Aut}{Aut}

\DeclareMathOperator{\Hom}{Hom} 
\DeclareMathOperator{\rank}{rk}

\DeclareMathOperator{\Gal}{Gal} 
 
\DeclareMathOperator{\Spec}{Spec}
 
\DeclareMathOperator{\lcm}{lcm} 
\DeclareMathOperator{\sgn}{sgn}

\renewcommand{\div}{\operatorname{div}}

\DeclareMathOperator{\End}{End}

\DeclareMathOperator{\Stab}{Stab}

\DeclareMathOperator{\Nm}{Nm}
\newcommand{\et}{\textnormal{\'et}}
\renewcommand{\subset}{\subseteq}

\newcommand{\nfield}[1]{\href{https://www.lmfdb.org/NumberField/#1}{\texttt{#1}}}
\usepackage{underscore} 
\newcommand{\avlink}[1]{\href{http://www.lmfdb.org/Variety/Abelian/Fq/#1}{\texttt{#1}}}

\makeatletter
\newcommand\newtag[2]{#1\def\@currentlabel{#1}\label{#2}}
\makeatother

\definecolor{wlabcol}{rgb}{0,0,0.3}

\setlist[itemize]{leftmargin=30pt, itemsep=2pt}
\setlist[enumerate]{leftmargin=30pt, itemsep=2pt}

\newcommand{\cond}{Assumption~\protect\hyperlink{assumption-star}{$(\star)$}}

\title[Galois groups of simple abelian varieties over finite fields and exceptional Tate classes]{Galois groups of simple abelian varieties over \\ finite fields and exceptional Tate classes}
\date{\today}

\author{Santiago Arango-Piñeros} 
\address{Department of Mathematics, Emory
University, Atlanta, GA 30322, USA}
\email{santiago.arango.pineros@gmail.com}
\urladdr{\url{https://sarangop1728.github.io/}}

\author{Sam Frengley} 
\address{School of Mathematics, University of Bristol, 
Bristol, BS8 1UG, UK}
\email{sam.frengley@bristol.ac.uk}
\urladdr{\url{https://samfrengley.github.io/}}

\author{Sameera Vemulapalli}
\address{Department of Mathematics,
  Harvard University}
\email{vemulapalli@math.harvard.edu}
\urladdr{\url{https://web.math.princeton.edu/~sameerav/}}

\usepackage{longtable}
\allowdisplaybreaks

\begin{document}

\begin{abstract}
We prove new cases of the Tate conjecture for abelian varieties over finite fields, extending previous results of Dupuy--Kedlaya--Zureick-Brown, Lenstra--Zarhin, Tankeev, and Zarhin. Notably, our methods allow us to prove the Tate conjecture in cases when the angle rank is non-maximal.

Our primary tool is a precise combinatorial condition which, given a geometrically simple abelian variety $A/\Fq$ with commutative endomorphism algebra, describes whether $A$ has  \emph{exceptional classes} (i.e., $\Gal(\Fqbar/\Fq)$-invariant classes in $H_{\et}^{2r}(A_{\Fqbar}, \QQ_\ell(r))$ not contained in the span of classes of intersections of divisors). The criterion depends only on the Galois group of the minimal polynomial of Frobenius and its action on the Newton polygon of $A$.

Our tools provide substantial control over the isogeny invariants of $A$, allowing us to prove a number of new results. Firstly, we provide an algorithm which, given a Newton polygon and CM field, determines if they arise from a geometrically simple abelian variety $A/\Fq$ and, if so, outputs one such $A$. As a consequence we show that every CM field occurs as the center of the endomorphism algebra of an abelian variety $A/\Fq$. Secondly, we refine a result of Tankeev and Dupuy--Kedlaya--Zureick-Brown on angle ranks of abelian varieties. In particular, we show that ordinary geometrically simple varieties of prime dimension have maximal angle rank. 
\end{abstract}

\maketitle

\vspace{-8mm}
\setcounter{tocdepth}{1}
\tableofcontents
\vspace{-11mm}

\section{Introduction}
\label{sec:intro}
Let $p$ be a prime number, let $q$ be a power of $p$,
and write $G_{\Fq} = \Gal(\Fqbar/\Fq)$ for the absolute Galois group of $\Fq$. Let $A/\Fq$ be a simple abelian variety of dimension $g > 0$, let $\End^0(A) \colonequals \End(A) \otimes \QQ$ denote the endomorphism algebra of $A$, and let $\mathcal{Z}^r(A)$ denote the free abelian group supported on ($\Fq$-rational) cycles of codimension $r$ on $A$. Let $\ell \neq p$ be a prime number and write $\QQ_\ell(r)$ for the $r^{\text{th}}$ Tate twist of $\QQ_\ell$. In \cite{tate-cycles-zeta} Tate made the following conjecture.

\begin{conjecture}[The Tate conjecture for $A$]
  \label{conj:tate}
  For each $1 \leq r \leq g$ the cycle homomorphism
  \begin{equation*}
    c^r \colon \mathcal{Z}^r(A) \otimes \QQ_\ell \to H_{\et}^{2r}(A_{\Fqbar}, \QQ_{\ell}(r))^{G_{\Fq}}
  \end{equation*}
  is surjective.  
\end{conjecture}

\subsection{The Tate conjecture and Galois groups}
\label{sec:tate-conjecture}
Tate~\cite{Tate1966} proved \Cref{conj:tate} for divisor classes (i.e., when $r = 1$). Subsequent works of Tankeev~\cite{Tankeev1984}, Zarhin~\cite{Zarhin91}, Lenstra--Zarhin~\cite{LenstraZarhin1993}, and Dupuy--Kedlaya--Zureick-Brown~\cite{DupuyKedlayaZureick-Brown22} leverage this to prove, in certain cases, that $H_{\et}^{2r}(A_{\Fqbar}, \QQ_\ell(r))^{G_{\Fq}}$ is generated by classes of intersections of divisors. We say a class in $H_{\et}^{2r}(A_{\Fqbar}, \QQ_\ell(r))^{G_{\Fq}}$ is \cdef{exceptional}\footnote{Many authors refer to an exceptional class as ``exotic''.} if is not contained in the span of the image of intersections of divisors under the cycle map $c^r$.

When $A$ is geometrically simple and has commutative endomorphism algebra $K_A = \End^0(A)$ we prove \Cref{thm:main-tate-thm}. This provides a simple combinatorial condition,
in terms of the action of the Galois group $G_A$ of the Galois closure of $K_A$ on the Newton slopes (as defined in \Cref{sec:notation}),
which exactly classifies when $H_{\et}^{2r}(A_{\Fqbar}, \QQ_\ell(r))^{G_{\Fq}}$ contains exceptional classes.

To illustrate the utility of \Cref{thm:main-tate-thm}, we provide a number of consequences. Namely, we reprove and generalize results of Lenstra--Zarhin, Tankeev, and Zarhin. Let $K_A^+ \subset K_A$ be the maximal totally real subfield and let $G_A^+$ be the Galois group of $K_A^+$.

\begin{theorem}
  \label{thm:others-coros}
  Let $A/\Fq$ be a geometrically simple abelian variety of dimension $g$ with commutative endomorphism algebra. Then, for each $1 \leq r \leq g$, the algebra $H_{\et}^{2r}(A_{\Fqbar}, \QQ_\ell(r))^{G_{\Fq}}$ contains no exceptional classes (and in particular the Tate conjecture holds for $A$) if any of the following holds:
  \begin{enumerate}[label=(\roman*)]
  \item \label{thm:tankeev-p}
    \textnormal{(Tankeev~{\cite[Theorem~1.2]{Tankeev1984}})}.
    The dimension $g$ is prime.
    
  \item \label{thm:zarhin}
    \textnormal{(Zarhin~\cite[Theorem~2.7.3]{Zarhin91})}.
    The abelian variety $A$ is of \emph{K3-type}, i.e., the Newton polygon of $A$ consists of a segment of length $1$ and slope $0$, a segment of length $2g - 2$ and slope $1/2$, and a segment of length $1$ and slope $1$. 
    \begin{equation*}
      [0, \underbrace{\tfrac{1}{2}, \dots, \tfrac{1}{2}}_{2g - 2}, 1]
    \end{equation*}
    
  \item \label{thm:lenstra-zarhin}
    \textnormal{(Lenstra--Zarhin~{\cite[Remarks~6.6~and~6.7]{LenstraZarhin1993}})}.
    The Newton polygon of $A$ has a segment of length $2$ and slope $1/2$ and all other slopes are contained in $\ZZ_{(2)}$.
    \begin{equation*}
       [\underbrace{\dots\dots\vphantom{\tfrac{1}{2}}}_{\in \ZZ_{(2)}}, \tfrac{1}{2}, \tfrac{1}{2}, \underbrace{\dots\dots\vphantom{\tfrac{1}{2}}}_{\in \ZZ_{(2)}}]
    \end{equation*}

  \item \label{thm:n-trans}
    There exists integers $d,n\geq 1$ such that the group $G_A^+$ is $d$-transitive, the Newton polygon of $A$ has a segment of length $2n$ and slope $1/2$ and all other slopes are contained in $\ZZ_{(2)}$, where $n \in \{ d, g-d \}$ and if $n$ is even then the dimension $g$ is odd.
    \begin{equation*}
      [\underbrace{\dots\dots\vphantom{\tfrac{1}{2}}}_{\in \ZZ_{(2)}}, \underbrace{\tfrac{1}{2}, \dots, \tfrac{1}{2}}_{2n}, \underbrace{\dots\dots\vphantom{\tfrac{1}{2}}}_{\in \ZZ_{(2)}}]
    \end{equation*}
        
  \item
     \label{thm:dvarepsilon-trans}
    There exists integers $d,n \geq 1$ and an odd integer $\varepsilon \geq 1$ such that the dimension $g$ is odd, the group $G_A^+$ is $d\varepsilon$-transitive, and the Newton polygon of $A$ has a segment of length $d\varepsilon$ and slope $n/\varepsilon$, a segment of length $d\varepsilon$ and slope $(\varepsilon - n)/\varepsilon$, and all other slopes are contained in $\ZZ_{(\varepsilon)}$.
    \begin{equation*}
      [\underbrace{\dots\dots\vphantom{\tfrac{1}{2}}}_{\in \ZZ_{(\varepsilon)}}, \underbrace{\tfrac{n}{\varepsilon},\dots,\tfrac{n}{\varepsilon}}_{d\varepsilon}, \underbrace{\dots\dots\vphantom{\tfrac{1}{2}}}_{\in \ZZ_{(\varepsilon)}}, \underbrace{\tfrac{\varepsilon-n}{\varepsilon},\dots,\tfrac{\varepsilon-n}{\varepsilon}}_{d\varepsilon}, \underbrace{\dots\dots\vphantom{\tfrac{1}{2}}}_{\in \ZZ_{(\varepsilon)}}]
    \end{equation*}
  \end{enumerate}
\end{theorem}

\begin{remark}
  \label{remark:notes-on-others-coros}
  We make the following notes on \Cref{thm:others-coros}:
  \begin{enumerate}[label=(\arabic*)]
  \item
    The claim in \Cref{thm:others-coros}\ref{thm:tankeev-p} is true even without the assumption that $A$ has commutative endomorphism algebra, as is proved in \cite{Tankeev1984}.
  \item
    A simple abelian variety whose Newton polygon is of the form in \Cref{thm:others-coros}\ref{thm:zarhin} or \ref{thm:lenstra-zarhin} is necessarily geometrically simple.
  \item \label{i:note-special-case} 
    In \Cref{thm:others-coros}\ref{thm:n-trans}, if $g$ is odd and $d = 2$ then the condition that $G_A^+$ is $d$-transitive may be replaced with the condition that $G_A^+$ to be primitive (i.e., $K_A^+$ has no nontrivial proper subfields).
  \item \label{i:note-less-restricitive}
    The claim in \Cref{thm:others-coros}\ref{thm:n-trans} is strictly more general than \ref{thm:zarhin} and \ref{thm:lenstra-zarhin} (those being the cases when $d = 1$). Moreover, in \ref{thm:n-trans},  the condition that $G_A^+$ is $d$-transitive may be replaced with the condition that there exists an element of $G_A^+$ inducing a bijection between any two $d$-element subsets (but possibly not respecting the order of the elements).
  \end{enumerate}
\end{remark}

To every simple abelian variety of dimension $\geq 2$ with commutative endomorphism algebra, we associate a \cdef{weighted permutation representation} (see \Cref{def:wpr}), which encodes both the Galois group $G_A$ and its action on the $q$-adic valuations of the Frobenius eigenvalues. With this language the conditions of \Cref{thm:main-tate-thm} are elementary (though computationally expensive) to check exhaustively for any given Newton polygon, simply by enumerating subgroups of $C_2 \wr S_g$. In \Cref{sec:admissible} we provide explicit criteria (see \Cref{prop:AV-is-admissible}) which preclude certain pairs of Galois groups and Newton polygons occurring together (in fact, we these conditions are sharp, see \Cref{prop:ord-igp}). Combining these results, and employing the computer algebra system~\texttt{Magma}~\cite{magma}, we prove \Cref{thm:exceptions} covering those $g \leq 6$.

\begin{theorem}
  \label{thm:exceptions}
  Let $A$ be a geometrically simple abelian variety of dimension $g \leq 6$ with commutative endomorphism algebra. Then the algebra $H_{\et}^{2r}(A_{\Fqbar}, \QQ_\ell(r))^{G_{\Fq}}$ contains an exceptional class for some $1 \leq r \leq g$ if and only if the weighted permutation representation associated to $A$ (see~\Cref{def:wpr}) is one of those listed in \textnormal{\cite{OurElectronic}}. Moreover, when $g \leq 5$, every weighted permutation representation listed in \textnormal{\cite{OurElectronic}} occurs for some abelian variety $A$ (explicit examples are listed in \Cref{app:tables}).
\end{theorem}

It is likely that when $g = 6$ every weighted permutation representation listed in \cite{OurElectronic} occurs for some abelian variety $A$. Indeed, as we discuss in \Cref{sec:inverse-galo-probl}, the only obstruction to constructing such an abelian variety is the existence of a CM field with prescribed Galois group and ramification data. In all cases but one this would follow from standard conjectures on the tame Grunwald problem (see \Cref{conj:tame-grunwald}). 

Let $\cR_A$ denote the set of $2g$ complex roots of the minimal polynomial the Frobenius endomorphism $\pi_A \colon A \to A$ (see~\Cref{sec:notation} for more details).

The idea of exploiting the Galois action on the Frobenius eigenvalues to study the Tate conjecture dates back to at least a 1981 letter from Serre to Ribet~\cite[pp.~6--11]{Serre2013} and Tankeev's work~\cite{Tankeev1984}. Tankeev's proof of the Tate conjecture for geometrically simple primefolds~\cite{Tankeev1984} relies on the fact that there exists a (nontrivial) relation of the form
\[
  \prod_{\pi \in S} \pi  = \zeta q^{\lvert S \rvert/2}
\]
for some root of unity $\zeta \in \Qbar$ and subset $S \subset \cR_A$ of even cardinality if and only if $A$ has exceptional classes (indeed, one can show that these classes occur in codimension $|S|/2$, see \Cref{lemma:exceptional-class}). This idea is also employed in \cite{Zarhin91} and \cite{LenstraZarhin1993}. In each case, the ``if'' direction of this statement is used to prove that, under various hypotheses, $A$ does not have exceptional classes.

We push this technique to its limit and set up combinatorial conditions so that in many cases (e.g., in dimensions $\leq 6$ in \Cref{thm:exceptions}), we may classify exactly when a geometrically simple abelian variety has exceptional classes.

\subsection{Angle ranks of abelian varieties}
\label{sec:angle-ranks-abelian}
For an abelian variety $A$, following \cite[Definition~1.2]{DupuyKedlayaZureick-Brown22}, let $\Gamma\!_A \subset \Qbar^\times$ be the subgroup generated by the \cdef{normalized Frobenius eigenvalues} $\pi / \sqrt{q}$ for $\pi \in \cR_A$ and define the \cdef{angle rank of $A$} to be $\delta_A \colonequals \dim_{\QQ}(\Gamma\!_A \otimes_\ZZ \QQ)$. 

Zarhin \cite[Remark~3.2(a)~and~Theorem~3.4.3]{Zarhin94} (see also \cite{Zarhin90,LenstraZarhin1993,Zarhin15}) has proved that some power $A^n$ of $A$ has an exceptional class if and only if the angle rank of $A$ is not maximal (i.e., $\delta_A < g$, equivalently, $A$ does not have a base change which is a ``neat'' abelian variety in the sense of Zarhin \cite[Section~3]{Zarhin94}). It is therefore notable that our techniques allow us to prove the Tate conjecture for many abelian varieties which \emph{do not have maximal angle rank}. For example, we have the following consequence of \Cref{thm:exceptions}.

\begin{coro}
  \label{coro:angle-corank-3}
  Consider the (isogeny class of an) abelian $6$-fold $A/\FF_3$ with minimal polynomial of Frobenius given by
  \begin{equation*}
    P_A(T) = T^{12} - 3 T^{11} + 14 T^9 - 21 T^8 - 27 T^7 + 120 T^6 - 81 T^5 - 189 T^4 + 378 T^3 - 729 T + 729.
  \end{equation*}
  Then $A$ has commutative endomorphism algebra, is geometrically simple, has angle rank $\delta_A = 3$, and has no exceptional Tate classes. In particular, the Tate conjecture holds for $A$.
\end{coro}

To the best of our knowledge, the example in \Cref{coro:angle-corank-3} is the first geometrically simple abelian variety with commutative endomorphism algebra and angle co-rank $g - \delta_A > 1$ for which the Tate conjecture is known. We discuss this example further in \Cref{sec:an-explicit-example}.

Many interesting criteria for an abelian variety to have maximal angle rank $\delta_A = g$ are given in \cite{DupuyKedlayaZureick-Brown22}. Our techniques allow us to refine some of these results, especially in the case when $A$ is geometrically simple. For example, we prove the following theorem, refining a result of Tankeev~\cite{Tankeev1984} (a further generalization is given in \Cref{thm:ord-tankeev}).

\begin{theorem}[Refined Tankeev's Theorem]
  \label{thm:ord-tankeev}
  Let $A$ be a simple abelian variety of prime dimension $g$ with commutative endomorphism algebra and suppose that $A$ is not supersingular. Then the angle rank of $A$ is either $1$, $g-1$, or $g$. Moreover:
  \begin{enumerate}[label=(\roman*)]
  \item \label{i:delta1-geom-simple}
    $\delta_A = 1$ if and only if $A$ is not geometrically simple, and
  \item \label{i:ar-max}
    if all the slopes of the Newton polygon of $A$ lie in $\ZZ_{(2)}$ (e.g., if $A$ is ordinary), then $\delta_A \in \{1, g\}$.
  \end{enumerate}
  In particular if $A$ is geometrically simple and all slopes of $A$ lie in $\ZZ_{(2)}$ then the Tate conjecture holds for all powers of $A$.
\end{theorem}

Another application of our results is the following. It was conjectured by Ahmadi--Shparlinski \cite[Section~5]{AhmadiShparlinsky10} that every geometrically simple ordinary Jacobian has maximal angle rank, after which counterexamples were provided in dimension $4$ in \cite{DupuyKedlayaRoeVincent21}. In spite of this, \Cref{thm:ord-tankeev} shows that the conjecture of Ahmadi--Shparlinksi holds for geometrically simple abelian varieties of prime dimension. 

\begin{question}
    Let $g$ be a positive composite integer, does there exist a geometrically simple ordinary abelian variety of dimension $g$ with angle rank less than $g$?
\end{question}

\subsection{Reconstruction theorems}
\label{sec:reconstr-algor}

In \Cref{sec:admissible}, we give necessary criteria for a geometrically simple abelian variety over $\Fq$ to have a certain Galois group and Newton polygon. We conjecture that these criteria are the only obstructions to the existence of such an abelian variety. 

We make progress towards this conjecture in the following way. Following \cite[Section~3.1]{us}, in \Cref{def:wpr} we associate to a given Galois group and Newton polygon a weighted permutation representation.We give \Cref{alg:wpr-to-weil}, which takes as input a suitable weighted permutation representation and an appropriate number field and outputs a $q$-Weil number realizing this data. This algorithm is practical for small $g$, and allows us to exhibit the abelian varieties in \Cref{thm:exceptions} and \Cref{coro:angle-corank-3}.

We prove that the existence of such an appropriate number field is equivalent to a ``strong inverse Galois problem'' for CM fields in which we impose local conditions at $p$. Thus, by running \Cref{alg:wpr-to-weil}, we prove in \Cref{prop:ord-igp} that the existence of a prime power $q = p^k$ and a geometrically simple abelian variety $A/\Fq$ with commutative endomorphism algebra and given weighted permutation representation is equivalent to this strong inverse Galois problem. This shows that previous conjectures (\cite[Conjecture~2.7]{DupuyKedlayaRoeVincent22} and \cite[Conjecture~7.1]{us}) follow, in the geometrically simple case, from a standard conjecture in algebraic number theory.

We employ the correctness of \Cref{alg:wpr-to-weil} (which we prove in \Cref{thm:alg-termination}) to deduce the following results.

\begin{theorem}
  \label{thm:CM-q-weil}
  Let $K/\QQ$ be a CM number field of degree $2g$ where $g$ is a prime number. The following two statements are equivalent:
  \begin{enumerate}[label=(\roman*)]
      \item \label{i:CM-q-weil-statement1} 
      there exists a prime power $q=p^k$ and an abelian variety $A/\Fq$ with commutative endomorphism algebra $\End^0(A) \cong K$, and
      \item either $K$ is not a bi-quadratic field, or $K$ is a bi-quadratic field containing $\QQ(i)$ or $\QQ(\zeta_3)$.
  \end{enumerate}
  Moreover, the abelian variety $A$ in \ref{i:CM-q-weil-statement1} may be taken to be ordinary, and can be taken to be geometrically simple if $K$ is not a bi-quadratic field. Conversely, if $K$ is a bi-quadratic field containing $\QQ(i)$ or $\QQ(\zeta_3)$, then any such $A$ is not geometrically simple.
\end{theorem}

\begin{theorem}
  \label{thm:gen-CM-q-Weil}
  Let $K/\QQ$ be a CM number field of degree $2g$. Then there exists a prime power $q = p^k$ and a geometrically simple abelian variety $A/\Fq$ such that the center of $\End^0(A)$ is isomorphic to $K$. Moreover $A$ may be chosen so that $\dim A \in \{g, 2g\}$.
\end{theorem}

Ideas similar to those in \Cref{alg:wpr-to-weil} have previously appeared in the literature. For example, see the work of Lenstra--Oort \cite[Section 3]{lenstra-oort} and Zarhin~\cite[Example~1.1.2 and Section~4]{Zarhin94}, \cite[Section~7]{Zarhin15} (the example in \cite{Zarhin94} is attributed to Lenstra). We are grateful to Everett Howe for communicating to us a preliminary version of this algorithm. In the \texttt{GitHub} repository~\cite{OurElectronic} we provide a reference implementation for \Cref{alg:wpr-to-weil} in \texttt{Magma}. 

\subsection{Outline of the paper}
\label{sec:outline-paper}
In \Cref{sec:weil-numbers}, we set up the language which we use to discuss $q$-Weil numbers throughout this paper. In particular, following \cite[Section~3]{us} we associate to every simple abelian variety $A/\Fq$ with commutative endomorphism algebra a \emph{weighted permutation representation} which encodes the Galois group of $A$ and its action on the Newton polygon of $A$. We prove a useful group theoretic formula which computes the prime factorization of the principal ideal generated by a $q$-Weil number (see \Cref{prop:div-infl}). We also introduce a number of important properties of weighted permutation representations such as \emph{admissibility}.

\Cref{sec:angle-ranks-geom} is devoted to proving \Cref{thm:ord-tankeev} and its generalization, \Cref{prop:refined-gen-tankeev}. To achieve this we show that knowing the weighted permutation representation associated to $A$ is sufficient to determine whether $A$ is geometrically simple (see \Cref{lemma:geom-irred}).

In \Cref{sec:an-expl-crit} we state \Cref{thm:main-tate-thm}. This gives an explicit criterion, in terms of the weighted permutation representation, for when a geometrically simple abelian variety with commutative endomorphism algebra admits
an exceptional Tate class. In \Cref{sec:applications-of-thm} we use \Cref{thm:main-tate-thm} to prove \Cref{thm:others-coros}. Building on techniques of Tankeev~\cite{Tankeev1984} and Zarhin~\cite{Zarhin91} we prove \Cref{thm:main-tate-thm} in \Cref{sec:proof-tate-thm}.

Finally in \Cref{sec:reconstr-algo,sec:inverse-galo-probl}, we study the reconstruction problems addressed by \Cref{thm:CM-q-weil,thm:gen-CM-q-Weil}. To this end, we provide \Cref{alg:wpr-to-weil}, which, given an appropriate number field and weighted permutation representation, outputs a $q$-Weil number with that weighted permutation representation. Employing the correctness and termination properties of this algorithm (which we prove in \Cref{thm:alg-termination}) we deduce \Cref{thm:CM-q-weil,thm:gen-CM-q-Weil}. We prove \Cref{thm:exceptions} by combining \Cref{thm:main-tate-thm,thm:alg-termination}.

\subsection{Notation and conventions}
\label{sec:notation}
Throughout this article we assume that $A/\Fq$ is a simple abelian variety. We write $\pi_A \colon A \to A$ for the $q$-Frobenius endomorphism of $A$. We write $K_A = \QQ(\pi_A) \subset \End^0(A)$ for the center of the endomorphism algebra $\End^0(A)$. Let $P_A(T) \in \ZZ[T]$ be the characteristic polynomial of $\pi_A$ acting on the $\ell$-adic Tate module $T_{\ell}(A) \colonequals \underset{\longleftarrow}{\lim} \;A[\ell^n]$. 

Let $v \colon \QQ^\times \to \ZZ$ be the $p$-adic valuation normalized so that $v(q) = 1$. We define the \cdef{$q$-Newton polygon} of $A$ to be the lower convex hull of the points $(i,v(a_i))$ where $P_A(T) = \sum_{i=0}^{2g} a_{i} T^{2g-i}$. By the \cdef{length} of a segment of a Newton polygon we mean the length of the segment when projected onto the horizontal axis.

\begin{defn}
  \label{def:assumption} \hypertarget{assumption-star}{}
  We say that an abelian variety $A$ as above satisfies \cdef{{\cond}} if the following conditions hold:
  \begin{enumerate}[label=(\Alph*)]
  \item \label{item:simple}
    $A$ is simple,
  \item \label{item:commutative-End}
    $\End^0(A) = K_A$ is commutative, and
  \item \label{item:K-CM}
    $K_A$ is a CM field.
  \end{enumerate}
\end{defn}

The conditions of {\cond}, together with the Honda--Tate theorem, imply that $P_A(T)$ is irreducible and that $K_A$ is a field of degree $2g$ with no real embeddings. Further, note that any simple abelian variety of dimension $g \geq 2$ with commutative endomorphism algebra satisfies \ref{item:K-CM} (and thus {\cond}). 

Let $A$ be an abelian variety satisfying {\cond}. Fix an algebraic closure $\Qbar$ of the field $K_A = \QQ(\pi_A)$ and let $L = L_A$ denote the Galois closure of $K_A$ in $\Qbar$. We define the \cdef{Frobenius eigenvalues of $A$} to be the roots of $P_A(T)$ in $L$ and write
\begin{equation}
  \label{eq:RA}
  \cR_A \colonequals \brk{\pi \in L : P_A(\pi) = 0}
\end{equation}
for the set of Frobenius eigenvalues. Note that since $K_A$ is a CM field, the elements of $\cR_A$ arise in complex conjugate pairs. That is, if $\pi \in \cR_A$ then its complex conjugate $\bar\pi = q/\pi$ is also contained in $\cR_A$.  We write $G = G_A$ for the Galois group of $L/\QQ$.
We define $K_A^+ = \QQ(\pi + \overbar\pi) \subset K_A$ to be the maximal totally real subfield of $K_A$ and write $G_A^+$ for the Galois group of the Galois closure of $K_A^+$ inside $L$.

For a number field $F/\QQ$ we write $\OO_F$ for the ring of integers of $F$ and let
\[
  \cP_F \colonequals \{ \mfP \in \Spec \OO_F : \mfP \mid p \}.
\]

If $G$ is a group and $\Omega$ is a $G$-set, we denote by $\QQ\ideal{\Omega}$ the free $\QQ$-module supported on $\Omega$, together with the induced $G$-action.

Finally recall that we say an element $\pi \in \overbar{\QQ}$ is a \cdef{$q$-Weil number} if the complex norm $| \iota(\pi) | = \sqrt{q}$ for every  embedding $\iota\colon \QQ(\pi) \hookrightarrow \CC$.

\subsection{Acknowledgments}
\label{sec:acknowledgements}
We are very grateful to Everett Howe for describing to us the outline of \Cref{alg:wpr-to-weil} and for several insightful correspondences. We also thank Melanie Matchett Wood and Julian Demeio for useful discussions and Taylor Dupuy and David Zureick-Brown for comments on an earlier version of this article. SF was funded by Céline Maistret's Royal Society Dorothy Hodgkin Fellowship.

\section{The Galois action on \texorpdfstring{$q$}{q}-Weil numbers}
\label{sec:weil-numbers} 
Let $A$ be a simple abelian variety. Fix a prime $\dpr \in \cP_L$ dividing $p$ and denote its decomposition group by $D$, its tame inertia group by $G_0$, and its wild inertia group by $G_1$.
We write $v_{\dpr}\colon L^\times \to \ZZ$ for the valuation corresponding to $\dpr$ (so that $x \in \dpr$ if and only if $v_{\dpr}(x) \geq 1$), and define $v$ to be its normalization
\begin{equation*}
    v(x) \colonequals \tfrac1k v_\dpr(x)
\end{equation*}
where $k \colonequals v_{\dpr}(q)$.

In this work, a significant role will be played by the factorization of the ideals $\pi \OO_L$ where $\pi$ ranges over the set $\cR_A$ of Frobenius eigenvalues. As we show in \Cref{prop:div-infl} this factorization allows us to describe many properties of $A$ using only its Galois group and Newton polygon, without direct reference to its Frobenius eigenvalues. To this end, we begin by recalling the divisor map defined in \cite[Section~3.3]{us}, which realizes the factorization of $\pi \OO_L$ as a morphism of $G$-representations.

\begin{defn}
  \label{def:divisor-map}
  Let $A$ be an abelian variety satisfying {\cond}. We define the \cdef{divisor map} $\div_A \colon \vecsp{\QQ}{\cR_A} \rightarrow \vecsp{\QQ}{\cP_L}$ to be the $\QQ$-linear map given by linearly extending the map ${\pi \mapsto \sum_{\mfP} a(\mfP) \mfP}$ where $\pi \in \cR_A$ and
  $\pi\cO_L = \prod_{\mfP} \mfP^{a(\mfP)}$ is the prime factorization of the ideal generated by $\pi$.   
\end{defn}

The following lemma is immediate from the construction.

\begin{lemma}
  \label{lemma:divA-G-map}
  The divisor map $\div_A$ is $G$-equivariant when $\cR_A$ and $\cP_L$ are equipped with the natural (left) $G$-actions.
\end{lemma}

\subsection{The Galois action on the Frobenius eigenvalues}

Consider the set of symbols
\begin{equation}
  \label{eq:X2g}
  X_{2g} \colonequals \brk{\decor{1}, \dots, \decor{g},\bar{\decor{g}}, \dots, \bar{\decor{1}}}.
\end{equation}
We define $W_{2g} \subset \mathrm{Sym}(X_{2g})$ to be the stabilizer of the partition
\begin{equation*}
  X_{2g} = \brk{\decor{1},\bar{\decor{1}}} \sqcup \dots \sqcup \brk{\g,\bar\g}.
\end{equation*}
We recall the following definitions from \cite[Section~3.1]{us}.

\begin{defn}
  \label{def:indexing}
  An \cdef{indexing of the roots} is a bijection $\mathcal{I}\colon X_{2d} \to \cR_A$ which satisfies the following conditions:
  \begin{enumerate}[label=(\roman*)]
  \item
    $\mathcal{I}$ respects complex conjugation, i.e., $\mathcal{I}(\overline{\decor{i}}) = \overline{\mathcal{I}(\decor{i})}$ for each $\decor{1} \leq \decor{i} \leq \g$, and
  \item
    the indices climb the Newton polygon, i.e.,
    \[
      v(\mathcal{I}(\decor{i})) \leq v(\mathcal{I}(\decor{j})) \leq v(\mathcal{I}(\bar{\decor{j}})) \leq v(\mathcal{I}(\bar{\decor{i}}))
    \]
    for each pair of indices $\1 \leq \decor{i} \leq \decor{j} \leq \decor{g}$.
  \end{enumerate}
\end{defn}

\begin{defn}
\label{def:slope-function}
A \cdef{weight function} is a map $w \colon X_{2g} \rightarrow \QQ_{\geq 0}$ that corresponds to a $q$-Newton polygon. That is:
\begin{enumerate}[label=(\roman*)]
    \item $w(\decor{i}) \leq w(\decor{j})$ for every $\decor{1}\leq \decor{i} \leq \decor{j} \leq \g$,
    \item $w(\decor{i}) + w(\bar{\decor{i}}) = 1$ for every $\decor{1}\leq \decor{i} \leq \g$, and
    \item every slope $s \in w(X_{2g})$ occurs with multiplicity divisible by the denominator of $s$.
\end{enumerate}
\end{defn}

\begin{defn}
  \label{def:wpr}
  A \cdef{weighted permutation representation} is a pair $\rho = (w,G)$ where $w \colon X_{2g} \rightarrow \QQ_{\geq 0}$ is a weight function and $G \subseteq W_{2g}$ is a transitive subgroup containing the complex conjugation element $(\1\bar{\1})\dots(\g\bar{\g})$. We say that $\rho$ is \cdef{supersingular} if $w(X_{2g}) = \{ 1/2 \}$ and \cdef{ordinary} if $w(X_{2g}) = \{0, 1\}$.
\end{defn}

\begin{remark}
  The weighted permutation representation (as defined in \Cref{def:wpr}) differs slightly from that defined in \cite[Section~3]{us} where the weight function is allowed to be an arbitrary map $w \colon X_{2g} \to \QQ_{\geq 0}$ and $G$ is not required to be transitive. The reader is also encouraged to compare our weighted permutation representation with the \emph{Newton hyperplane representation} considered by Dupuy--Kedlaya--Zureick-Brown~\cite{DupuyKedlayaZureick-Brown22}.
\end{remark}

Let $A$ be an abelian variety satisfying {\cond}, let $G = G_A$, and choose an indexing $\mathcal{I} \colon X_{2g} \to \cR_A$. Write $\pi_{\decor{i}} = \mathcal{I}(\decor{i})$ and $\bar\pi_{\i} = \mathcal{I}(\bar \i)$. Choose as usual, a prime $\dpr$ above $p$ in $L$. The (normalized $\dpr$-adic) valuation $v\colon L^\times \to \QQ$ gives a weight function $w = v \circ \mathcal{I}$ and the identification $\mathcal{I}$ makes $X_{2g}$ into a $G$-set, thus inducing an inclusion $G \hookrightarrow W_{2g} \subset \operatorname{Sym}(X_{2g})$. We therefore obtain a weighted permutation representation $\rho = (w, G)$ associated to $A$ (and $\mathcal{I}$). This identification also yields an embedding $G^+ \hookrightarrow \operatorname{Sym}(\{\1, \dots \g \})$.

Note that a different choice of indexing $\mathcal{I}' \colon X_{2g} \to \cR_A$ yields a different inclusion $G \hookrightarrow W_{2g}$, whose image is conjugate by an element of the stabilizer of the weight function $w$, i.e., the subgroup $\Stab(w) = \{ \sigma \in W_{2g} : w(\sigma(x)) = w(x) \text{ for all } x \in X_{2g}\}$. In particular the weighted permutation representation associated to $A$ is well defined up to $\Stab(w)$-conjugation of $G$.

\subsection{A necessary condition for $\rho$ to arise from an abelian variety}
\label{sec:admissible}
We now show that if a weighted permutation representation arises from an abelian variety, it must satisfy some constraints arising from the factorization of $p$ in $L$. In our previous work~\cite[Section~4--6]{us} we exploited similar restrictions to rule out certain weighted permutation representations in dimensions $1$, $2$, and $3$. Recall that we have fixed a prime ideal $\dpr \in \cP_L$ and thus a normalized valuation $v \colon L^\times \to \QQ$. We write $D$, $G_0$, and $G_1$ for the decomposition, tame inertia, and wild inertia groups of $\dpr$ respectively. 

\begin{lemma}
  \label{prop:decomp-group}
  Let $A$ be an abelian variety satisfying {\cond}. Let $G \supseteq D \supseteq G_0 \supseteq G_1$ be as above. Then there exists a surjection $\Gal(\overbar{\QQ}_p/\QQ_p) \to D$ such that the image of the tame inertia subgroup is $G_0$ and the image of the wild inertia subgroup is $G_1$. Moreover:
  \begin{enumerate}[label=(\roman*)]
  \item \label{i:NP-stab} 
    for all $\pi \in \cR_A$ and for all $\sigma \in D$, we have $v(\pi) = v(\sigma\pi)$, and
  \item \label{i:vertical-length} 
    for every $\pi \in \cR_A$, we have $\sum_{\pi' \in D\pi}v(\pi') \in \ZZ$.
  \end{enumerate}  
\end{lemma}

\begin{proof}
  The claim in \ref{i:NP-stab} is a property of the decomposition group (see e.g., \cite[IV.2]{Serre67}). Part \ref{i:vertical-length} is a consequence of {\cond} part \ref{item:commutative-End} and the fact that $q$-Newton polygons of Frobenius polynomials of abelian varieties have lattice points as break points.
\end{proof}

Motivated by \Cref{prop:decomp-group}, we make the following definition which provides a constraint on the weighted permutation representations which may arise from abelian varieties over $\Fq$ (satisfying {\cond}).

\begin{defn}
  \label{def:p-admissible-filtration}
  Let $p$ be a prime number and let $\rho = (w, G)$ be a weighted permutation representation. We say that a sequence of subgroups
  \[G \supseteq D \supseteq G_0 \supseteq G_1\]
  is a \cdef{weak $p$-admissible filtration of $\rho$} if $D \subseteq \Stab(w)$ and there exists a surjection $\Gal(\overbar{\QQ}_p/\QQ_p) \to D$ such that the image of the inertia group is $G_0$ and the image of the wild inertia group is $G_1$. We say that $\rho$ is \cdef{weakly $p$-admissible} if it admits a weak $p$-admissible filtration.

  A \cdef{strong $p$-admissible filtration of $\rho$} is a weak $p$-admissible filtration which enjoys the property that for every $\decor{i} \in X_{2d}$, we have $\sum_{\decor{j} \in D\decor{i}}w(\decor{j}) \in \ZZ$. We say that $\rho$ is \cdef{strongly $p$-admissible} if it admits a strong $p$-admissible filtration.
\end{defn}

\begin{proposition}
  \label{prop:AV-is-admissible}
  Suppose that $A/\Fq$ is an abelian variety satisfying {\cond} which (together with a choice of indexing) gives rise to a weighted permutation representation $\rho$. Then $\rho$ is strongly $p$-admissible (where $p$ is the characteristic of $\Fq$).
\end{proposition}

\begin{proof}
  The claim follows directly from the assertions listed in \Cref{prop:decomp-group}.
\end{proof}

\begin{remark}
  \label{remark:conditions-on-D}
  Observe that if $\rho = (w, G)$ is weakly $p$-admissible, then the explicit presentation of the tame part of $\Gal(\overbar{\QQ}_p/{\QQ}_p)$ given by a theorem of Iwasawa (see \cite[Theorem~7.5.3]{neukirch-cohomology}) imposes the following conditions:
  \begin{enumerate}[label=(\roman*)]
  \item
    $G_0 \subseteq D$ is a normal subgroup and $D/G_0$ is cyclic,
  \item
    $G_1 \subseteq D$ is a normal subgroup and $G_0/G_1$ is cyclic of order prime to $p$,
  \item
    $G_1$ is a $p$-group, and
  \item
    there exist elements $\sigma, \tau \in D/G_1$ such that $\sigma$ and $\tau$ generate $D/G_1$, $\tau$ generates $G_0/G_1$ and $\sigma \tau \sigma^{-1} = \tau^p$.
  \end{enumerate}
  The reader is encouraged to compare these conditions with those in \cite[Theorem~1.3]{Liu-inertia}, \cite[Theorem~1.3]{Liu-inertia-thesis}. When $p$ is odd an explicit presentation for $\Gal(\overbar{\QQ}_p/\QQ_p)$ was computed by Jannsen and Wingberg and is given in \cite[Theorem~7.5.14]{neukirch-cohomology}. In principle, one can deduce from this theorem the precise group theoretic conditions classifying triples $(D,G_0,G_1)$ arising from surjections from $\Gal(\overbar{\QQ}_p/\QQ_p)$. It would be interesting to use this to compute a list of weakly $p$-admissible weighted permutation representations. In this paper, we do not make use of this explicit presentation and instead show weak $p$-admissibility by exhibiting an appropriate number field.
\end{remark}

\subsection{A group-theoretic formula for \texorpdfstring{$\div_A$}{div\_A}}
Let $A$ be an abelian variety satisfying {\cond} and fix an indexing of the roots and a prime $\dpr$ in $L$ above $p$. In this section, we show that the divisor map, $\div_A$, can be realized directly from the weighted permutation representation associated to $A$. Set $K \coloneqq \QQ(\pi_{\decor{1}})$ and let $H \subseteq G$ be the subgroup fixing $K \subseteq L$. The following lemma follows by construction together with standard results.

\begin{lemma}
  \label{lemma:g-isos}
  If $A$ satisfies {\cond} then we have isomorphisms of $G$-sets
  \begin{enumerate}[label=(\roman*)]
    \item \label{lemma:G-mod-H-eq-X2g} 
    $G/H \xrightarrow{\sim} X_{2g}$ given by $\sigma H \mapsto \sigma(\decor{1})$, and
    \item \label{lemma:G-mod-D-eq-PL}
    $G/D \xrightarrow{\sim} \cP_L$ given by $\sigma D \mapsto \sigma(\dpr)$.
  \end{enumerate}
\end{lemma}

We emphasize that in \Cref{defn:Phi} and \Cref{lemma:Phi-Gequiv} the weighted permutation representation $\rho = (w,G)$ need not arise from an abelian variety.

\begin{defn}
  \label{defn:Phi}
  Given a weighted permutation representation $\rho = (w,G)$ define the $\QQ$-linear map $\Phi_{\rho} \colon \vecsp{\QQ}{X_{2g}} \rightarrow \vecsp{\QQ}{G}$ to be given by
  \begin{equation*}
    \Phi_{\rho}(\decor{i}) \colonequals \sum_{\sigma \in G}w(\sigma^{-1}(\decor{i}))\sigma.
  \end{equation*}
\end{defn}

\begin{lemma}
  \label{lemma:Phi-Gequiv}
  If $\rho= (w,G)$ is a weighted permutation representation then the $\QQ$-linear map $\Phi_{\rho} \colon \vecsp{\QQ}{X_{2g}} \rightarrow \vecsp{\QQ}{G}$ is $G$-equivariant. 
\end{lemma}
\begin{proof}
  For each $\tau \in G$, we have
  \begin{equation*}
    \tau\Phi_{\rho}(\decor{i}) = \sum_{\sigma \in G}w(\sigma^{-1}(\decor{i}))\tau \sigma = \sum_{\gamma \in G}w(\gamma^{-1} \tau(\decor{i}))\gamma  
  \end{equation*}
  which is equal to $\Phi_{\rho}(\tau (\decor{i}))$, as required.
\end{proof}

Consider the inflation map $\inf \colon \vecsp{\QQ}{G/D} \rightarrow \vecsp{\QQ}{G}$ given by $\sigma D \mapsto \sum_{\tau \in D}\sigma \tau$ and identify the $G$-representations $\vecsp{\QQ}{\cP_L}$ and $\vecsp{\QQ}{G/D}$ via the isomorphism in \Cref{lemma:g-isos}\ref{lemma:G-mod-D-eq-PL}.

\begin{proposition}
  \label{prop:div-infl}
  Let $A$ be an abelian variety satisfying {\cond}, let $k = v_{\dpr}(q)$, and let $\rho = \rho_A = (w,G)$ be a weighted permutation representation associated to $A$ (for some choice of indexing). Then the following diagram commutes:
  \begin{equation*}
    \begin{tikzcd}
      &  \vecsp{\QQ}{G/D} \arrow{dr}{\inf} \\
      \vecsp{\QQ}{X_{2g}} \arrow{ur}{\div_A} \arrow[rr, "k\cdot\Phi_{\rho}"'] &&  \vecsp{\QQ}{G}.
    \end{tikzcd}
  \end{equation*}
\end{proposition}

\begin{proof}
  Because $G$ is transitive and all maps in the triangle are $G$-equivariant, it suffices to show that $\inf \div_A (\decor{1}) = k\cdot\Phi_{\rho}(\decor{1})$. For each $\sigma \in G$ let $a(\sigma) \in \ZZ_{\geq 0}$ be chosen so that 
  \begin{equation*}
    \pi_{\decor{1}} \cO_L = \prod_{\sigma D \in G/D}\sigma(\dpr)^{a(\sigma)}.
  \end{equation*}
  Therefore, applying \Cref{lemma:g-isos}\ref{lemma:G-mod-D-eq-PL}, for all $\tau \in G$ we have
  \begin{align*}
    \label{eqn:expansion-of-taupi1}
    \tau \pi_{\decor{1}} \cO_L &= \prod_{\sigma D \in G/D} \tau \sigma(\dpr)^{a(\sigma)} \\
                               &= \prod_{\gamma D \in G/D} \gamma(\dpr)^{a(\tau^{-1} \gamma)} .
  \end{align*}
  Taking valuations on both sides, we see $v_{\dpr}(\tau \pi_{\decor{1}}) = a(\tau^{-1})$. It follows that for each $\sigma \in G$ we have $a(\sigma) = v_{\dpr}(\sigma^{-1} \pi_{\decor{1}}) = k v(\sigma^{-1} \pi_{\decor{1}}) = k w(\sigma^{-1} (\decor{1}))$. Therefore
  \begin{equation*}
    \pi_{\decor{1}} \cO_L = \left( \prod_{\sigma D \in G/D}  \sigma(\dpr)^{w(\sigma^{-1}(\decor{1}))} \right)^k.
  \end{equation*}
  By \Cref{lemma:g-isos}\ref{lemma:G-mod-D-eq-PL} and the definition of $\div_A$ we see that
  \begin{equation*}
    \div_A (\decor{1}) = k \sum_{\sigma D \in G/D} w(\sigma^{-1}(\decor{1})) \; \sigma D \in \vecsp{\QQ}{G/D},
  \end{equation*}
  and therefore
  \begin{equation*}
    \inf \div_A (\decor{1}) =  k \sum_{\sigma \in G} w(\sigma^{-1}(\decor{1})) \; \sigma \in \vecsp{\QQ}{G},
  \end{equation*}
  which is equal to $k \cdot \Phi_{\rho}(\decor{1})$, as required.
\end{proof}

\section{Angle ranks of geometrically simple abelian varieties}
\label{sec:angle-ranks-geom}

In this section we prove \Cref{thm:ord-tankeev}. The first input is \Cref{lemma:geom-irred}, in which we show that if $A$ is an abelian variety satisfying {\cond} then we may recognize whether $A$ is geometrically simple directly from its associated weighted permutation representation. Indeed it is shown in \cite[Lemma~3.11]{us} that the angle rank of $A$ may also be determined from its weighted permutation representation. Combining these results with ideas of Dupuy--Kedlaya--Zureick-Brown~\cite[Section~3]{DupuyKedlayaZureick-Brown22} (and refinements thereof) we prove \Cref{thm:ord-tankeev}. More precisely, we prove a generalization of \Cref{thm:ord-tankeev}, namely \Cref{prop:refined-gen-tankeev}.

\begin{defn}
  We say that a weighted permutation representation $\rho = (w,G)$ is \cdef{geometrically simple} if $\Stab_G(\Phi_{\rho}(\decor{1})) = \Stab_G(\decor{1})$ where $\Phi_\rho$ is the linear map defined in \Cref{defn:Phi}.
\end{defn}

\begin{lemma}
  \label{lemma:geom-irred}
  Let $A$ be an abelian variety satisfying {\cond} and let $\rho$ be a weighted permutation representation associated to $A$ (for some choice of indexing). Then $A$ is geometrically simple if and only if $\rho$ is geometrically simple. 
\end{lemma}
\begin{proof}
    Observe that in general, $\Stab_G(\1) \subset \Stab_G(\div_A(\1))$. By \Cref{prop:div-infl}, $\Stab_G(\div_A(\1)) = \Stab_G(\Phi_{\rho}(\decor{1}))$, so it suffices to show that $A$ is geometrically simple if and only if $\Stab_G(\1) = \Stab_G(\div_A(\1))$.
    
    By the Honda--Tate theorem, $A$ is geometrically simple if and only if the minimal polynomial of the Frobenius of the base change $A_{\FF_{q^n}}$ is irreducible for every $n \geq 1$. Thus $A$ is geometrically simple if and only if for all integers $n \geq 1$ and $\decor{1} \leq \decor{i} < \decor{j} \leq \decor{g}$ we have $\pi_{\decor{i}}^n \neq \pi_{\decor{j}}^n$. But, since $\pi_{\decor{i}} / \pi_{\decor{j}}$ has norm $1$ for all places of $L$, this is the case if and only if $\pi_{\decor{i}} \cO_L \neq \pi_{\decor{j}} \cO_L$ for all $\decor{1} \leq \decor{i} < \decor{j} \leq \decor{g}$. This observation, combined with the transitivity of $G$, implies that $A$ is geometrically simple if and only if $\Stab_G(\1) = \Stab_G(\div_A(\1))$, as required.
\end{proof}

We now recall from \cite[Section~3]{us} that the angle rank of $A$ may be computed from its associated weighted permutation representation. To this end we make the following definition (which can be seen to be equivalent to that given in~\cite[Definition~3.8]{us} by combining \Cref{prop:div-infl} with \cite[Lemma~3.11]{us}, cf.~\cite[Proposition~3.1]{DupuyKedlayaZureick-Brown22}).

\begin{defn}
  We define the \cdef{angle rank} of a weighted permutation representation $\rho$ to be
  \[\delta_\rho = \rank \Phi_\rho - 1.\]
\end{defn}

\begin{lemma}
  \label{lemma:ar-lemma}
  Let $A$ be an abelian variety satisfying {\cond} and let $\rho$ be a weighted permutation representation associated to $A$ (for some choice of indexing of the roots). Then $\delta_\rho = \delta_A$.
\end{lemma}

\begin{proof}
  This is \cite[Lemma~3.11]{us}.
\end{proof}

Before proving \Cref{thm:ord-tankeev} we introduce some notation. We write
\begin{equation*}
  \label{eq:QQ-X+}
  \vecsp{\QQ}{X_{2g}}^+ \colonequals  \frac{ \vecsp{\QQ}{X_{2g}} }{ \langle \decor{1} + \overline{\decor{1}}, \dots, \decor{g} + \overline{\decor{g}} \rangle},
\end{equation*}
\begin{equation*}
  \label{eq:QQ-G+}
  \vecsp{\QQ}{G}^+ \colonequals \frac{ \vecsp{\QQ}{G} }{ \big\langle \textstyle\sum_{\sigma \in G}\sigma \big\rangle},
\end{equation*}
and note that there is an induced $G$-equivariant linear map
\begin{equation}
  \label{eq:Phi+}
  \Phi_{\rho}^+ \colon \vecsp{\QQ}{X_{2g}}^+ \rightarrow \vecsp{\QQ}{G}^+ . 
\end{equation}
In particular, the angle rank $\delta_\rho$ of $\rho$ is equal to $\rank \Phi_\rho^+$. 

Let $\rho = (w,G)$ be a weighted permutation representation (not necessarily arising from an abelian variety). There is a natural exact sequence 
\[
    0 \to (\ZZ/2\ZZ)^g \to W_{2g} \to S_g \to 0.
\]
The group $(\ZZ/2\ZZ)^g$ is naturally identified with $\Hom(\{\1,\dots,\g\}, \ZZ/2\ZZ)$. Restricting to $G \subseteq W_{2g}$, we obtain an exact sequence
\[
    0 \to C \to G \to G^+ \to 0
\]
and therefore an inclusion $C \hookrightarrow \Hom(\{\1,\dots,\g\}, \ZZ/2\ZZ)$. By \cite[Lemma~3.10]{DupuyKedlayaZureick-Brown22} there exists a unique \cdef{level set partition} 

\[
  \{ \1, \dots ,\g\} = T_1 \sqcup \dots \sqcup T_m
\]
such that each element $c \in C$ is constant on each $T_i$ (when considered as a function on $\{\1,\dots,\g\}$), and the sets $T_i$ are maximal with this property. The first part of the following proposition is precisely \cite[Theorem 1.7]{DupuyKedlayaZureick-Brown22}, which states that if $A$ is an abelian variety and satisfies {\cond} (so $A$ is not supersingular), and $g/m$ is prime, then $A$ has angle rank $\delta_A \in \{m,g-m,g\}$. Indeed, the reader should compare the proof of \Cref{prop:refined-gen-tankeev} with \cite[Section~3.3]{DupuyKedlayaZureick-Brown22}.

Further, observe that if an abelian variety $A$ satisfies {\cond} and if $m = g$, then we have $\delta_A = g$ (by \cite[Corollary~3.12]{DupuyKedlayaZureick-Brown22}). In particular, \Cref{thm:ord-tankeev} follows immediately from \Cref{prop:refined-gen-tankeev} since $m \in \{1, g\}$ if $g$ is prime.

\begin{proposition}
  \label{prop:refined-gen-tankeev}
  Let $\rho = (w, G)$ be a weighted permutation representation (not necessarily arising from an abelian variety) which is not supersingular and which has $m$ elements in its level set partition. Suppose that $\tfrac{g}{m}$ is prime. Then $\rho$ has angle rank $\delta_\rho \in \{m,g-m,g\}$. Moreover, if $\tfrac{g}{m}$ is odd then:
  \begin{enumerate}[label=(\roman*)]
  \item \label{i:delta1-geom-simple-general} 
    $\delta_\rho = m < g$ if and only if $\rho$ is not geometrically simple, and
  \item \label{i:ar-max-general}
    if $w(X_{2g}) \subset \ZZ_{(2)}$, then $\delta_\rho \in \{m, g\}$.
  \end{enumerate}
\end{proposition}

\begin{proof}
  Let $H_1 \colonequals \Stab_G(T_1)$. Note that $H_1$ is a cyclic subgroup of order $\tfrac{g}{m}$ which acts on the component $T_1$ of the level set partition. Consider the vector subspace $\vecsp{\QQ}{T_1} \subset \vecsp{\QQ}{X_{2g}}$, and write $\vecsp{\QQ}{T_1}^+$ for its image in $\vecsp{\QQ}{X_{2g}}^+$.  Let $\Phi_1^+$ denote the restriction of $\Phi_\rho^+$ to $\vecsp{\QQ}{T_1}^+$. The kernel of $\Phi_1^+$ is an $H_1$-representation. Since $\tfrac{g}{m}$ is prime, $\vecsp{\QQ}{T_1}^+$ decomposes into the direct sum of the diagonal subspace and the trace-zero subspace. In particular, we have $\dim(\ker \Phi_1^+) \in \{0, 1, \tfrac{g}{m}-1,\tfrac{g}{m}\}$. Now, as argued in \cite[Lemma 3.11]{DupuyKedlayaZureick-Brown22} we have
  \[
    \delta_\rho = m\cdot \rank \Phi_1^+ =  m \left(\dim \QQ\ideal{T_1}^+ - \dim(\ker \Phi_1^+) \right) = g - m\dim(\ker \Phi_1^+).
  \]
  Since $\delta_\rho = 0$ if and only if $\rho$ is supersingular we may further assume that $\dim(\ker \Phi_1^+) \in \{0, 1, \tfrac{g}{m}-1\}$.

  To see \ref{i:delta1-geom-simple-general} note that if $\delta_\rho = m$, then $\dim (\ker \Phi_1^+) = \tfrac{g}{m} - 1$ and thus $\rank(\Phi^+_1) = 1$. This implies that $\rho$ is not geometrically simple. Conversely, if $\rho$ is not geometrically simple, then $\Stab_G(\Phi_{\rho}(\1))$ strictly contains $\Stab_G(\1)$, so $\rank( \Phi^+_{1} ) \leq \lvert T_1\rvert/2$. But $|T_1| = \tfrac{g}{m} > 2$ (since $g$ is odd), so $\dim (\ker \Phi_1^+) = \tfrac{g}{m} - 1$, as required.

  To see \ref{i:ar-max-general}, suppose in search of contradiction that $\rho$ has angle rank $g-1$ and therefore $\dim(\ker \Phi_1^+) = 1$. Let $\sum_{\i}a_\i \,\i$ be a preimage in $\vecsp{\QQ}{T_1}$ for a generator of $\ker (\Phi_1^+)$.  Since $\ker \Phi_1^+$ is $H_1$-stable for all $\sigma \in H_1$ we have
  \[
    \sigma\bigg(\sum_{\i \in S_1}a_{\i} \,\i \bigg) = \pm \sum_{\i}a_{\i} \, \i,
  \]
  and because $H_1$ acts transitively on $\vecsp{\QQ}{T_1}$ we may assume $a_{\i} \in \{-1,1\}$. Set
  \begin{equation*}
    T^+ \colonequals \{\i \in T_1 : a_{\i} = 1\} \quad \text{and} \quad T^- \colonequals \{\i \in T_1 : a_{\i} = -1\}.
  \end{equation*}
  Since $T_1 = T^+ \sqcup T^-$, the difference $\lvert T^+\rvert - \lvert T^-\rvert$ is odd. By construction, we have
  \[
    \sum_{\i \in T^+}\Phi_{\rho}(\i) - \sum_{\i \in T^-}\Phi_{\rho}(\i) = \frac{\lvert T^+\rvert - \lvert T^-\rvert}{2}\sum_{\sigma \in G}\sigma.
  \]
  For each $\sigma \in G$ the coefficient of $\sigma$ in above expression is 
  \[
    \sum_{\i \in T^+}w(\sigma^{-1}(\i)) - \sum_{\i \in T^-}w(\sigma^{-1}(\i)) =  \frac{\lvert T^+\rvert - \lvert T^-\rvert}{2}
  \]
  However, the right hand side is not in $\ZZ_{(2)}$ and by the assumption on $w$, the left hand side is in $\ZZ_{(2)}$, a contradiction.
\end{proof}

We record the following interesting consequence of \Cref{prop:refined-gen-tankeev}.

\begin{coro}
  Let $A$ be a geometrically simple abelian variety of dimension $g$ satisfying {\cond}. Suppose that $|C| > 2$, that the dimension of $A$ is a product of two (not necessarily distinct) odd primes, and that the slopes of the Newton polygon of $A$ all lie in $\ZZ_{(2)}$. Then $A$ has maximal angle rank $\delta_A = g$. In particular, all powers of $A$ satisfy the Tate conjecture.
\end{coro}

\section{An explicit criterion for the existence of exceptional classes}
\label{sec:an-expl-crit}
We now state \Cref{thm:main-tate-thm}, which provides a criterion for the existence of exceptional classes in $H_\et^{2r}(A_{\Fqbar}, \QQ_\ell(r))$ for geometrically simple abelian varieties $A$ satisfying {\cond}. Our result is given in terms of the weighted permutation representation $\rho = (w, G)$ associated to $A$. To this end, we make the following definition.

\begin{defn}
  \label{defn:exceptional-wpr}
  For sets $T^+, T^{-} \subset \{\decor{1}, \dots \decor{g}\}$ we write
  \begin{equation}
    \label{eq:exceptional-thing}
    \exceptional{\rho}{T^+}{T^-} \colonequals \sum_{\decor{i} \in T^+}\Phi_{\rho}(\decor{i}) - \sum_{\decor{j} \in T^-}\Phi_{\rho}(\decor{j}).
  \end{equation}
  We say that a geometrically simple weighted permutation representation $\rho = (w,G)$ is \cdef{exceptional} if there exist disjoint sets $T^+, T^- \subseteq \{\decor{1},\dots,\g\}$ such that $|T^+| + |T^-|$ is even, $T^+ \sqcup T^{-}$ is nonempty, and
  \begin{equation}
    \label{eq:exceptional-wpr-cond}
    \exceptional{\rho}{T^+}{T^-} = \frac{|T^+| - |T^-|}{2} \sum_{\sigma \in G} \sigma.
  \end{equation}
\end{defn}

We remark that \eqref{eq:exceptional-wpr-cond} is equivalent to the statement that for all $\sigma \in G$, we have
\begin{equation}
  \label{eqn:exceptional-criterion}
  \frac{|T^+| - |T^-|}{2} = \sum_{\i \in T^+}w(\sigma^{-1}(\i)) - \sum_{\j \in T^-}w(\sigma^{-1}(\j)).
\end{equation}

\begin{lemma}
  \label{lem:nonempty-Splus-Sminus}
  Let $\rho = (w,G)$ be a geometrically simple exceptional weighted permutation representation and let $T^+,T^- \subset \{ \1, \dots \g\}$ be the sets in \Cref{defn:exceptional-wpr} realizing the exceptionality of $\rho$. Then $T^+$ and $T^-$ are both nonempty.    
\end{lemma}
\begin{proof}
By symmetry, it suffices to show that $T^-$ is nonempty. Suppose for sake of contradiction that $T^-$ is empty. Note that $\rho$ is not supersingular (i.e., $w$ is not the constant function taking value $1/2$) since it is geometrically simple and $g > 1$. Taking $\sigma \in G$ to be the complex conjugation element $(\1\bar{\1})\dots(\g\bar{\g})$, we see that 
\[
    \sum_{\i \in T^+}w(\i) - \sum_{\j \in T^-}w(\j) = \sum_{\i \in T^+}w(\i) < \sum_{\i \in T^+}w(\overline{\i}) = \sum_{\i \in T^+}w(\sigma^{-1}(\i)) -  \sum_{\j \in T^-}w(\sigma^{-1}(\j)).
\]
However, \eqref{eqn:exceptional-criterion} implies that the left hand side equals the right hand side, a contradiction.
\end{proof}

Before stating \Cref{thm:main-tate-thm}, we rephrase the condition of a weighted permutation representation arising from an abelian variety being exceptional in terms of a criterion on the normalized Frobenius eigenvalues.

\begin{lemma}
  \label{lemma:wpr-exceptional-lambdas}
  Let $A$ be a geometrically simple abelian variety satisfying {\cond}. Fix an indexing $\mathcal{I}$ of the Frobenius eigenvalues and write $\lambda_\i = \pi_\i / \sqrt{q}$ for the corresponding normalized Frobenius eigenvalues. The weighted permutation representation $\rho$ associated to the pair $(A, \mathcal{I})$ is exceptional if and only if there exists a multiplicative relation
  \begin{equation}
    \label{eq:lam-rel}
    \prod_{\decor{i} \in T^+} \lambda_{\decor{i}} \prod_{\decor{j} \in T^-} \lambda_{\decor{j}}^{-1} = \zeta
  \end{equation}
  where $\zeta \in \Qbar$ is a root of unity, and $T^+, T^- \subset \{ \decor{1}, \ldots, \decor{g} \}$ satisfy the conditions of \Cref{defn:exceptional-wpr}.
\end{lemma}

\begin{proof}
  Let $\pi_{\1}, \dots, \pi_{\g}, \bar{\pi}_{\g}, \dots, \bar{\pi}_{\1}$ be the Frobenius eigenvalues of $A$. A relation of the form \eqref{eq:lam-rel} is equivalent to 
  \[
    \prod_{\i \in T^+}\pi_\i \prod_{\j \in T^-}\pi_\j^{-1} = \zeta q^{(\lvert T^+ \rvert - \lvert T^- \rvert)/2}.
  \]

  Applying the divisor map, and noting that $q = \pi_{\i} \bar{\pi}_\i$, we see that \eqref{eq:lam-rel} is equivalent to
  \[
    \sum_{\i \in T^+}\div_A(\pi_\i) - \sum_{\j \in T^-}\div_A(\pi_\j) = \frac{\lvert T^+ \rvert - \lvert T^- \rvert}{2}\div_A(q).
  \]
  Now, composing with the inflation map and applying \Cref{prop:div-infl}, we see that \eqref{eq:lam-rel} is equivalent to 
  \[
     \sum_{\i \in T^+}\Phi_{\rho}(\i) - \sum_{\j \in T^-}\Phi_{\rho}(\j) = \frac{\lvert T^+ \rvert - \lvert T^- \rvert}{2}\sum_{\sigma \in G}\sigma, 
  \]
  as required.
\end{proof}

It may be the case that when applying \Cref{thm:main-tate-thm} for an explicit abelian variety $A$, checking the conditions on \Cref{lemma:wpr-exceptional-lambdas} is easier than checking those in \Cref{defn:exceptional-wpr}; see \Cref{remark:computations}. We stress, however, that phrasing \Cref{thm:main-tate-thm} in terms of the weighted permutation representation provides a very convenient tool for proving results about \emph{families} of abelian varieties (as in \Cref{thm:others-coros,thm:exceptions}). The proof of \Cref{thm:main-tate-thm} is given in \Cref{sec:proof-tate-thm}.

\begin{theorem}
  \label{thm:main-tate-thm} 
  Let $\ell\neq p$ be a prime number, and let $A/\Fq$ be a geometrically simple abelian variety satisfying {\cond} and giving rise to a weighted permutation $\rho$. If $\rho$ is not exceptional then for all $1 \leq r \leq g$ the $\QQ_\ell$-vector space $H_{\et}^{2r}(A_{\Fqbar}, \QQ_\ell(r))^{G_{\Fq}}$ is generated by intersections of divisor classes, and therefore the Tate conjecture holds for $A$.

  Conversely, if $\rho$ is exceptional, then $A$ has exceptional classes. Moreover, these classes occur in $H_{\et}^{2r}(A_{\Fqbar}, \QQ_\ell(r))^{G_{\Fq}}$ whenever $2r \geq |T^+| + |T^-|$ where $T^+,T^- \subset \{ \decor{1}, \dots, \decor{g} \}$ are as in \Cref{defn:exceptional-wpr}.
\end{theorem}

\begin{remark}
  \label{remark:computations}
  In principle, on input of an abelian variety $A/\Fq$ one can check the conditions of \Cref{thm:main-tate-thm} on a computer. In practice, the bottleneck in applying \Cref{thm:main-tate-thm} directly is computing the Galois group $G$ and its action on $p$-adic approximations to the roots of $P_A(T)$ (though one may not need to compute the entire splitting field). To be certain that the criteria in \Cref{thm:main-tate-thm} are satisfied, up to numerical tolerance, one can apply the LLL algorithm to complex approximations to the normalized Frobenius eigenvalues $\lambda \colonequals \pi / \sqrt{q}$ with $\pi \in \cR_A$ to determine all relations of the form given in \Cref{lemma:wpr-exceptional-lambdas} (a similar approach is outlined in \cite[Section~3.8]{DupuyKedlayaRoeVincent22} for computing angle ranks). We provide a \texttt{Magma} implementation of this numerical algorithm in \cite{OurElectronic}.
\end{remark}

\subsection{Applications of \texorpdfstring{\Cref{thm:main-tate-thm}}{Theorem 4.4}}
\label{sec:applications-of-thm}
Before proceeding with the proof of \Cref{thm:main-tate-thm}, we deduce \Cref{thm:others-coros,thm:exceptions} from the introduction. \Cref{thm:others-coros} includes results of Tankeev~\cite[Theorem~1.2]{Tankeev1984}, Zarhin~\cite[Theorem~2.7.3]{Zarhin91}, and Lenstra--Zarhin~\cite[Main Theorem~5.5]{LenstraZarhin1993}. Our arguments are heavily influenced by theirs, and the reader is encouraged to compare both approaches.

For $\sigma \in G$ and subset $\Lambda \subset X_{2g}$ we write
\begin{equation}
  \label{eq:gamma-sigma}
  \Gamma_{\sigma}(\Lambda) = \{\i \in \{\1,\dots,\g\} : \sigma^{-1}(\decor{i}) \in \Lambda \}.
\end{equation}

\begin{proof}[Proof of \Cref{thm:others-coros}\ref{thm:tankeev-p}]
  When $g = 2$ the claim was proved by Tate~\cite{Tate1966}, so we may assume that $A/\Fq$ is a geometrically simple abelian variety of odd prime dimension $g$ satisfying {\cond}. Following \cite[Corollary~3.6]{DupuyKedlayaZureick-Brown22}, fix an indexing and consider the map $\Phi_\rho^+ \colon \vecsp{\QQ}{X_{2g}}^+ \to \vecsp{\QQ}{G}^+$ as in \eqref{eq:Phi+}. By \Cref{prop:refined-gen-tankeev} the 
  dimension of the kernel of $\Phi_\rho^+$ is either $0$ or $1$.
  
  If $\dim(\ker \Phi_\rho^+) = 0$ then $\rho$ is not exceptional (in fact, this is the case when $A$ has maximal angle rank and by \cite[Theorem 3.4.3]{Zarhin94} no power of $A$ has an exceptional class).
  
  It remains to consider the case where $\dim(\ker \Phi_{\rho}^+) = 1$. Suppose in search of contradiction that $\rho$ is exceptional and choose sets $T^+,T^-$ satisfying \Cref{defn:exceptional-wpr}. Since $\ker \Phi_\rho^+$ has dimension $1$ we have an equality elements of $\vecsp{\QQ}{X_{2g}}^+$ given by
  \[
    \sigma \left( \sum_{\i \in T^+}\i - \sum_{\j \in T^-}\overline{\j} \right) = \pm \sum_{\i \in T^+}\i - \sum_{\j \in T^-}\overline{\j}.
  \]
  for all $\sigma \in G$.
  
  This implies that $G^+$ fixes the (nonempty) set $T^+ \sqcup T^-$. Because $G$ acts transitively on $X_{2g}$, this implies that $T^+ \sqcup T^- = \{\1,\dots,\g\}$. This is a contradiction since $g$ is odd and the cardinality of $T^+ \sqcup T^-$ is even. Therefore $\rho$ is not exceptional and the claim follows from \Cref{thm:main-tate-thm}.
\end{proof}

\begin{proof}[Proof of \Cref{thm:others-coros}\ref{thm:zarhin}--\ref{thm:n-trans}]
  We prove \Cref{thm:others-coros}\ref{thm:n-trans}, noting that \ref{thm:zarhin} and \ref{thm:lenstra-zarhin} follow from the case when $d = 1$. Let $\mathfrak{s} = \{\g - \decor{n}, \dots, \g\}$, write $\overline{\mathfrak{s}} \subseteq \{\overline{\g - \decor{n}},\dots,\overline{\g}\}$ and write $\widetilde{\mathfrak{s}} = \mathfrak{s} \sqcup \overline{\mathfrak{s}}$. Let $w$ be the associated weight function given by
  \begin{equation*}
    w(x) =
    \begin{cases}
      \tfrac{1}{2}  & \text{if $x \in \widetilde{\mathfrak{s}}$, and} \\
      \in \ZZ_{(2)} & \text{otherwise.}
    \end{cases}
  \end{equation*}  
  Consider the weighted permutation representation $\rho = (w, G)$, where $G \subset W_{2g}$ is a transitive subgroup containing complex conjugation.
  
  As in \Cref{remark:notes-on-others-coros}\ref{i:note-less-restricitive}, suppose that for any two $d$-element subsets $\Gamma, \Gamma' \subset \{\1, \dots, \g\}$ there exists $\sigma \in G^+$ such that $\sigma(\Gamma) = \Gamma'$ (this is a strictly weaker condition than $d$-transitivity). Indeed, taking complements, the same is true for any two $n$-element subsets $\Gamma, \Gamma' \subset \{\1, \dots, \g\}$.

  By \Cref{thm:main-tate-thm}, it suffices to show that $\rho$ is not exceptional. In particular it suffices to show that there do not exist $T^+, T^- \subset \{ \decor{1}, \dots \decor{g} \}$ as in \Cref{defn:exceptional-wpr}. Suppose for the sake of contradiction that such sets $T^+$ and $T^-$ do exist and let $T = T^+ \sqcup T^-$.

  Note that $|T|$ is even and nonzero by assumption. For all $\sigma \in G$ the coefficient of $\sigma$ in $\exceptional{\rho}{T^+}{T^-}$ is equivalent modulo $\ZZ_{(2)}$ to
  \begin{equation}
    \label{eqn:Z-breaker}
    \sum_{\decor{i} \in T^+} w(\sigma^{-1}(\decor{i})) - \sum_{\decor{j} \in T^-} w(\sigma^{-1}(\decor{j})) \equiv
     \sum_{\i \in \Gamma_\sigma(\widetilde{\mathfrak{s}}) \cap T} \frac{1}{2}\; .
  \end{equation}

  We may assume that $g > n$, otherwise $A$ is supersingular. Let $\Gamma_\sigma(\widetilde{\mathfrak{s}})$ be as defined in \eqref{eq:gamma-sigma}. Since the expression in \eqref{eqn:Z-breaker} is an integer, for  all $\sigma \in G$ we have $| \Gamma_{\sigma}(\widetilde{\mathfrak{s}}) \cap T | \in 2\ZZ$. By the assumption on $G^+$ every $n$-element subset of $\{\1,\dots,\g\}$ arises as $\Gamma_{\sigma}(\widetilde{\mathfrak{s}})$ for some $\sigma \in G$. Combining these two statements we see that any $n$-element subset $\Gamma \subset \{\1,\dots,\g\}$ must have $| \Gamma \cap T | \in 2\ZZ$.

  First suppose that $n$ is odd. If $|T| > n$ then take $\Gamma \subset T$ to be any $n$-element subset, in which case $|\Gamma \cap T|$ is odd, a contradiction. Therefore, since $|T|$ is even, we have $|T| < n < g$. Because $T$ is nonempty we may choose an element $\i \in T$ and a subset $\Gamma' \subset \{1,\dots,\g\} \setminus T$ of cardinality $|\Gamma'| = n+1-|T|$. Taking $\Gamma = \Gamma' \cup T \setminus \{\i\}$ we have $| \Gamma \cap T| = |T| - 1$ is odd, again a contradiction.

  It remains to consider the case when $n$ is even and $g$ is odd. Since $g$ is odd and $T$ is nonempty and of even cardinality we may choose a $n$-element subset $\Upsilon \subset \{\1, \dots, \g\}$ such that
  $\Upsilon$ is not a subset of $T$ and $\Upsilon \cap T \neq \emptyset$. Choose a distinct $n$-element set $\Upsilon' \neq \Upsilon$ and take $\i \in \Upsilon' \setminus \Upsilon$. If $\i \in T$, then choose an element $\j \in \Upsilon \setminus T$ and consider the $n$-element set $\Gamma = \{\i\} \cup \Upsilon \setminus \{\j\}$. If $\i \notin T$, then choose an element $\j \in \Upsilon \cap T$, and consider the $n$-element set $\Gamma = \{\i\} \cup \Upsilon \setminus \{\j\}$. In either case $| \Gamma \cap T |$ is odd, a contradiction.

  We now prove \ref{thm:n-trans} when $d = 2$ and $g$ is odd, but under the weaker assumption that $G^+$ is primitive, thus proving \Cref{remark:notes-on-others-coros}\ref{i:note-special-case}. In this case let $\Gamma_\sigma'(\widetilde{\mathfrak{s}}) = \Gamma_{\sigma}(\widetilde{\mathfrak{s}})$ if $n = d$, respectively $\Gamma_\sigma'(\widetilde{\mathfrak{s}}) = \{\1, \dots \g\} \setminus \Gamma_{\sigma}(\widetilde{\mathfrak{s}})$ if $n = g - d$. Then  $\lvert \Gamma_{\sigma}'(\widetilde{\mathfrak{s}})\rvert = 2$, and thus $\lvert \Gamma_{\sigma}'(\widetilde{\mathfrak{s}}) \cap T\rvert$ is either $0$ or $2$ (since $g$ is odd and \eqref{eqn:Z-breaker} is integral). Therefore $\Gamma_{\sigma}'(\widetilde{\mathfrak{s}})$ nontrivially intersects $T$ if and only if it is contained in $T$. Now, consider a graph $\Delta$ whose vertex set is $\{\1,\dots,\g\}$ and such that there is an edge between $\i,\j$ if and only if $\{\i,\j\} = \Gamma_{\sigma}'(\widetilde{\mathfrak{s}})$ for some $\sigma \in G$. Then $T$ is a union of vertex sets of connected components of $\Delta$. The graph $\Delta$ admits an action of $G^+$ so, because $G^+$ is primitive, $\Delta$ must be connected. Therefore $T = \{\1,\dots,\g\}$, which is a contradiction when $g$ is odd.
\end{proof}

\begin{proof}[Proof of \Cref{thm:others-coros}\ref{thm:dvarepsilon-trans}]
  Let $\mathfrak{s} \subset \{\1,\dots,\g\}$ and $\overbar{\mathfrak{s}} \subset \{\bar\1, \dots \bar\g\}$ be defined so that the weight function associated to $A$ is
  \begin{align*}
    w(x) =
    \begin{cases}
      n/\varepsilon                 & \text{if $x \in \mathfrak{s}$,}   \\
      (\varepsilon - n)/\varepsilon & \text{if $x \in \overbar{\mathfrak{s}}$,} \\
      \in \ZZ_{(\varepsilon)}       & \text{otherwise.}
    \end{cases}
  \end{align*}
  
  As before, consider a geometrically simple weighted permutation representation $\rho = (w, G)$ where $G \subseteq W_{2g}$ contains complex conjugation. By \Cref{thm:main-tate-thm}, it suffices to show that $\rho$ is not exceptional. In particular it suffices to show that there do not exist $T^+, T^- \subset \{ \decor{1}, \dots \decor{g} \}$ as in \Cref{defn:exceptional-wpr}. Suppose for the sake of contradiction that they do exist. 

  Let $T = T^+ \sqcup T^-$ and note that $|T|$ is even and nonzero by assumption. For all $\sigma \in G$ the coefficient of $\sigma$ in $\exceptional{\rho}{T^+}{T^-}$ is equal modulo $\ZZ_{(\varepsilon)}$ to 
  \begin{equation}
    \label{eqn:breaking-n-eps}
    \sum_{\i \in \Gamma_{\sigma}(\mathfrak{s}) \cap T^+} \frac{n}{\varepsilon}  + \sum_{\i \in \Gamma_{\sigma}(\overline{\mathfrak{s}}) \cap T^+} \frac{\varepsilon - n}{\varepsilon} - \Bigg(\sum_{\j \in \Gamma_{\sigma}(\mathfrak{s}) \cap T^-} \frac{n}{\varepsilon}  + \sum\limits_{\j \in \Gamma_{\sigma}(\overline{\mathfrak{s}}) \cap T^-} \frac{\varepsilon - n}{\varepsilon} \Bigg) .
  \end{equation}

  The expression in \eqref{eqn:breaking-n-eps} must be integral. In particular, because $\gcd(n,\varepsilon) = 1$, for all $\sigma \in G$ we have
  \begin{equation}
    \label{eqn:equiv}
    | \Gamma_{\sigma}(\mathfrak{s}) \cap T^+ | - | \Gamma_{\sigma}(\overbar{\mathfrak{s}}) \cap T^+ | - | \Gamma_{\sigma}(\mathfrak{s}) \cap T^- | + | \Gamma_{\sigma}(\overbar{\mathfrak{s}}) \cap T^- | \equiv 0 \pmod{\varepsilon}.
  \end{equation}
  
  Let $C = \ker(G \rightarrow G^+)$. If $|C| > 2$, then the result follows from \cite[Theorem~1.7]{DupuyKedlayaZureick-Brown22}. Thus we may assume that $C$ is the subgroup of order two generated by the ``complex conjugation'' element $(\1\bar{\1})\dots(\g\bar{\g}) \in G$. Consider the exact sequence
  \[
    0 \rightarrow C \rightarrow G \rightarrow G^+ \rightarrow 0\; .
  \]
  Because $C$ is contained in the center of $G$, the action of $G^+$ on $C$ is trivial. Therefore we have isomorphisms $G \cong C \rtimes G^+ \cong C \times G^+$ and the exact sequence splits. Choose a splitting $G^+ \xhookrightarrow{} G$.

  The data of the lift $G^+ \xhookrightarrow{} G$ provides a labelling $\{\i^{\vee},\i^{\wedge} \} = \{\i, \overline{\i}\}$, so that the action of $G^+$ preserves the partition
  \[
    X_{2g}  = \{\1^\vee,\dots,\g^\vee\} \sqcup \{\1^\wedge,\dots,\g^\wedge\}.
  \]
  Define a function $\sgn \colon \{\1,\dots,\g\} \rightarrow \{\pm 1\}$ by
  \[
    \sgn(\i) =
    \begin{cases}
      1 & \text{if $\i = \i^\vee$} \\
      -1 & \text{if $\i = \i^\wedge$}
    \end{cases}
  \]
  
  By assumption $G^+$ is $d\varepsilon$-transitive and therefore, since \eqref{eqn:equiv} is integral, for any $d\varepsilon$-element subset $\Gamma \subset \{\1,\dots,\g\}$ and any bijection $\alpha \colon \Gamma \rightarrow \mathfrak{s}$, we have

  \begin{equation}
    \label{eqn:equiv-2}
    |\Gamma^+ \cap T^+| + |\Gamma^- \cap T^-| \equiv |\Gamma^- \cap T^+| + |\Gamma^+ \cap T^-| \pmod{\varepsilon}
  \end{equation}
  where
  \begin{equation*}
      \Gamma^+ = \{\i \in \Gamma  :  \sgn(\i) = \sgn(\alpha(\i))\}
      \quad \text{and} \quad
      \Gamma^- = \{\i \in \Gamma  :  \sgn(\i) \neq \sgn(\alpha(\i))\}.
  \end{equation*}
  We now break into two cases, depending on the cardinality of $\sgn(\mathfrak{s})$.\\

  \noindent \textbf{Case 1:} $| \sgn(\mathfrak{s}) | =2$. \\
  Note that by \Cref{lem:nonempty-Splus-Sminus}, both $T^+$ and $T^-$ are nonempty. Choose $\i \in T^+$ and $\j \in T^-$  and choose $\i',\j' \in \mathfrak{s}$ with $\sgn(\i') \neq \sgn(\j')$. Let $\Gamma$ be a $d\varepsilon$-element set containing both $\i$ and $\j$ and consider two bijections $\alpha,\beta \colon \Gamma \rightarrow \mathfrak{s}$ such that
  \[
    \alpha(\i) = \i', \alpha(\j) = \j' \qquad \text{and} \qquad  \beta(\i) = \j', \beta(\j) = \i'.
  \]
  If \eqref{eqn:equiv-2} is satisfied for $\alpha$, it cannot be satisfied for $\beta$, a contradiction. \\

  \noindent \textbf{Case 2:} $|\sgn(\mathfrak{s})|=1$. \\
  Set $\Sigma \coloneqq \{1,\dots,\g\}$ and define
  \[
    \Sigma_1 \coloneqq \big( T^+ \cap \sgn^{-1}(1) \big) \sqcup \big( T^- \cap \sgn^{-1}(-1) \big),
  \]
  \[
    \Sigma_2 \coloneqq \big( T^+ \cap \sgn^{-1}(-1) \big) \sqcup \big( T^- \cap \sgn^{-1}(1) \big),
  \]
  and $\Sigma_3 \coloneqq \Sigma \setminus (\Sigma_1 \sqcup \Sigma_2)$ so that $\Sigma = \Sigma_1 \sqcup \Sigma_2 \sqcup \Sigma_3$.
  
  In this case, \eqref{eqn:equiv-2} is equivalent to the statement that every subset $\Gamma \subseteq \Sigma$ of cardinality $d\varepsilon$ satisfies:
  \begin{equation}
    \label{eqn:equiv-3}
    \lvert \Sigma_1 \cap \Gamma \rvert \equiv \lvert \Sigma_2 \cap \Gamma \rvert \pmod{\varepsilon}.
  \end{equation}

  We claim that $\Sigma = \Sigma_a$ for some $a \in \{1,2,3\}$. Indeed, suppose that $\Sigma \neq \Sigma_a$ for any $a$ and take a subset $\Gamma \subset \Sigma$ of cardinality $d\varepsilon$ not contained in any $\Sigma_i$. Choose $\i \notin \Gamma$ and $\j \in \Gamma$ which are contained in distinct components of the partition $\Sigma = \Sigma_1 \sqcup \Sigma_2 \sqcup \Sigma_3$. Taking $\Gamma' \coloneqq \{\i\} \cup (\Gamma \setminus \{\j\})$ and noting that $| \Sigma_1 \cap \Gamma | \equiv | \Sigma_2 \cap \Gamma| \pmod{\varepsilon}$, we have $| \Sigma_1 \cap \Gamma' | \not \equiv | \Sigma_2 \cap \Gamma'| \pmod{\varepsilon}$, contradicting the assumption that $\Sigma \neq \Sigma_a$ for some $a \in \{1,2,3\}$.

  If $\Sigma = \Sigma_{1}$ or $\Sigma_2$, then $T = \{\1,\dots,\g\}$ contradicting the assumption that $g$ is odd and $T$ has even cardinality. If $\Sigma = \Sigma_3$ then $T$ is empty, again a contradiction.
\end{proof}

\section{The proof of \texorpdfstring{\Cref{thm:main-tate-thm}}{??}}
\label{sec:proof-tate-thm}
Let $A$ be an abelian variety satisfying {\cond} and let $\agrp{A} \subset \Qbar^\times$ be the subgroup generated by the normalized Frobenius eigenvalues $\pi / \sqrt{q}$ where $\pi$ ranges over those $\pi \in \cR_A$. 

Our approach and presentation follows that taken by Tankeev~\cite{Tankeev1984} in proving the Tate conjecture for geometrically simple primefolds over $\Fq$. An analogous approach is taken by Zarhin~\cite{Zarhin91} and Lenstra--Zarhin~\cite[Section~5]{LenstraZarhin1993} which is given in the dual language of multilinear algebra on the Tate module $T_\ell(A) \otimes \QQ$.  Many of these ideas can also be found in a motivic language in Milne's work~\cite[Theorem~7.1]{Milne99} proving that the Hodge conjecture for CM abelian varieties over $\CC$ implies the Tate conjecture for abelian varieties over $\Fq$.

\begin{lemma}
  \label{lemma:isog-inv}
  The existence of exceptional classes is invariant under isogeny and passage to finite extensions. More precisely, let $B \to A$ be an isogenous abelian variety, and let $\kappa/\Fq$ be a finite extension. We have the following:
  \begin{enumerate}[label=(\roman*)]
  \item \label{enum:isog}
    \textnormal{(Isogeny invariance).}
    The space $H_\et^{2r}(A_{\Fqbar}, \QQ_\ell(r))^{G_{\Fq}}$ contains an exceptional class if and only if $H_\et^{2r}(B_{\Fqbar}, \QQ_\ell(r))^{G_{\Fq}}$ contains an exceptional class.
  \item \label{enum:bc}
    \textnormal{(Base change invariance).}
    The space $H_\et^{2r}(A_{\Fqbar}, \QQ_\ell(r))^{G_{\Fq}}$ contains an exceptional class if and only if $H_\et^{2r}(A_{\Fqbar}, \QQ_\ell(r))^{G_{\kappa}}$ contains an exceptional class.
  \end{enumerate}
\end{lemma}

\begin{proof}
  The isogeny invariance claimed in \ref{enum:isog} follows from the fact that the $G_{\Fq}$-representations $H_{\et}^{2r}(A_{\Fqbar}, \QQ_\ell(r))$ and $H_{\et}^{2r}(B_{\Fqbar}, \QQ_\ell(r))$ are isomorphic, and this isomorphism is given on $c^r(\mathcal{Z}^r(A))$ by precomposition with pullbacks of cycles.

  For \ref{enum:bc} we follow the sketch given in \cite[Section~2]{Totaro}. Suppose that $H_\et^{2r}(A_{\Fqbar}, \QQ_\ell(r))^{G_{\kappa}}$ is spanned by intersections of divisors and consider an element $\xi \in H_\et^{2r}(A_{\Fqbar}, \QQ_\ell(r))^{G_{\Fq}}$. By assumption $\xi$ is in the class of $c^r(\alpha)$ where $\alpha \in \mathcal{Z}^r(A_\kappa)$ is an intersection of divisors and thus $\boldsymbol{\alpha} \colonequals \sum_{\sigma \in \Gal(\kappa/\Fq)} \sigma(\alpha)$ is an element of $\mathcal{Z}^r(A)$. Moreover $\boldsymbol{\alpha}$ is an intersection of divisors on $A$ and $c^r(\boldsymbol{\alpha}) = [\kappa : \Fq] c^r(\alpha)$. 
\end{proof}

By replacing $\Fq$ with a finite extension and applying \Cref{lemma:isog-inv}\ref{enum:bc} we may assume without loss of generality that $q$ is a square and that $\agrp{A}$ is torsion free. Since $q$ is a square, there exists a normalized Frobenius endomorphism $\lambda_A = \pi_A / \sqrt{q} \in K_A$. 

By \Cref{lemma:isog-inv} and a theorem of Honda (see~\cite[Théorème~2]{Tate1971} or \cite[Theorem 4.1.1]{chai-conrad-oort}) we may assume without loss of generality that $A$ is the special fiber of an abelian scheme $\mathcal{A}/R$ where $R$ is the localization of the ring of integers of a finite extension of $F/\QQ$ at a prime above $p$ and the geometric fiber $\mathcal{A}_\CC$ is an abelian variety of CM type $(K_A, \Psi)$. More precisely, $\Psi = (\psi_{\decor{i}})_{\decor{i} = \decor{1}}^{\decor{g}}$ where each $\psi_{\decor{i}}$ is a complex embedding of $K_A$ and $\psi_{\decor{i}} \neq \psi_{\decor{j}}, \overbar{\psi}_{\decor{j}}$ for all $\decor{i} \neq \decor{j}$, and $\mathcal{A}(\CC) \cong \CC^{g} / \Psi(\mathfrak{a})$ for some fractional $\OO_{K_A}$-ideal $\mathfrak{a}$. In particular, $\End^0(\mathcal{A}_{\CC}) \cong K_A$. Again applying \Cref{lemma:isog-inv}\ref{enum:bc} and replacing $F$ with a finite extension we may assume without loss of generality that $\End^0(\mathcal{A}_{\CC}) \cong \End^0(\mathcal{A}_{F})$.

\begin{lemma}
  \label{lemma:decomp}
  Fix an isomorphism $\overbar{\QQ}_\ell \cong \CC$. We have a $K_A$-algebra isomorphism
  \begin{equation*}
    H_\et^1(A_{\Fqbar}, \QQ_\ell(1)) \otimes \overbar{\QQ}_\ell \cong V \oplus \overbar{V} 
  \end{equation*}
  where $V = \bigoplus_{\i=1}^\g V_\i$ (respectively $\overbar{V} = \bigoplus_{\i=1}^\g \overbar{V}_\i$) and the normalized Frobenius $\lambda_A = \pi_A / \sqrt{q} \in K_A$ acts on the $1$-dimensional vector spaces $V_\i$ (resp. $\overbar{V_\i}$) by multiplication with the normalized Frobenius eigenvalue $\lambda_{\decor{i}} = \psi_{\decor{i}}(\lambda_A)$ (resp. $\bar{\lambda}_{\decor{i}} = \overbar{\psi}_{\decor{i}}(\lambda_A)$).
\end{lemma}

\begin{proof}
  By flat base change we have an isomorphism $H_\et^1(A_{\Fqbar}, \ZZ_\ell) \cong H_\et^1(\mathcal{A}_{\overbar{\QQ}_p}, \ZZ_\ell)$ of representations of $G_{\Fq}$. In particular, upon tensoring with $\overbar{\QQ}_\ell$ we have an isomorphism of $K_A$-algebras $H_\et^1(A_{\Fqbar}, \ZZ_\ell) \otimes \overbar{\QQ}_\ell \cong T_\ell(\mathcal{A}_{\CC})^* \otimes \overbar{\QQ}_\ell$ where $T_\ell(\mathcal{A}_\CC)^* = \Hom_{\ZZ_\ell}(T_\ell(\mathcal{A}_\CC), \ZZ_\ell(1))$. The claim follows by noting that $T_\ell(\mathcal{A}_\CC) \otimes \overbar{\QQ}_\ell \cong \Psi(\mathfrak{a}) \otimes \CC$ as $K_A$-algebras. 
\end{proof}

For a subset $S \subset \{ \decor{1}, ..., \decor{g} \}$ write $V_{S} = \bigotimes_{\decor{i} \in S} V_{\decor{i}}$ and $\overbar{V}_S = \bigotimes_{\decor{i} \in S} \overbar{V}_{\decor{i}}$. 

\begin{lemma}
  \label{lemma:wedge-iso}
  For each $1 \leq r \leq g-1$ we have a $K_A$-algebra isomorphism
  \begin{align*}
    H_\et^{2r}(A_{\Fqbar}, \QQ_\ell(r)) \otimes \overbar{\QQ}_\ell &\cong \bigwedge\nolimits^{2r} (V \oplus \overbar{V})  \\
        &\cong \bigoplus_{S^+,S^-} V_{S^+} \otimes \overbar{V}_{S^-}.
  \end{align*}
  where the sum ranges over subsets $S^+, S^-  \subset \{\decor{1}, \dots, \decor{g}\}$ such that $|S^+| + |S^-| = 2r$.
\end{lemma}

\begin{proof}
  First recall that $H_\et^{2r}(A_{\Fqbar}, \QQ_\ell) \cong \bigwedge^{2r} H_\et^1(A_{\Fqbar}, \QQ_\ell)$. Combining this with \Cref{lemma:decomp} yields the first isomorphism. The second isomorphism follows by direct calculation.
\end{proof}

\begin{lemma}
  \label{lemma:wedge-time}    
  The subspace of the exterior algebra $\bigwedge^* H_\et^2(A_{\Fqbar}, \QQ_\ell(1))$ spanned by intersections of $\Fq$-rational divisors is the span of exterior products of elements of $\bigoplus_{\i=1}^\g V_\i \otimes \overbar{V}_\i$.
\end{lemma}

\begin{proof}
  Tate~\cite{Tate1966} proved that $H_\et^2(A_{\Fqbar}, \QQ_\ell(1))^{G_{\Fq}}$ is spanned by divisor classes. Note that the intersection of cycles corresponds to exterior multiplication in $\bigwedge^* H_\et^2(A_{\Fqbar}, \QQ_\ell(1))$. By \Cref{lemma:wedge-iso} we have
  \begin{equation*}
    H_\et^2(A_{\Fqbar}, \QQ_\ell(1))^{G_{\Fq}} \cong \bigoplus_{S^+,S^-} \left( V_{S^+} \otimes \overbar{V}_{S^-} \right)^{\lambda_A}
  \end{equation*}
  where the sum ranges over subsets $S^+, S^-  \subset \{\decor{1}, \dots, \decor{g}\}$ such that $|S^+| + |S^-| = 2$. However, since $\lambda_{\decor{i}}^{-1} = \bar{\lambda}_{\decor{i}}$, a summand $V_{S^+} \otimes \overbar{V}_{S^-}$ is fixed by $\lambda_A$ if and only if $S^+ = S^- = \{ \decor{i} \}$.
\end{proof}

\begin{lemma}
  \label{lemma:exceptional-class}
  The space $H_\et^{2r}(A_{\Fqbar}, \QQ_\ell(r))$ contains an exceptional class if and only if there exists a relation
  \begin{equation*}
    \prod_{\decor{i} \in T^+} \lambda_{\decor{i}} \prod_{\decor{j} \in T^-} \lambda_{\decor{j}}^{-1} = 1
  \end{equation*}
  where $T^+, T^- \subset \{ \decor{1}, \ldots, \decor{g} \}$ satisfy the conditions of \Cref{defn:exceptional-wpr} and $|T^+| + |T^-| \leq 2r$.
\end{lemma}

\begin{proof}
  Let $S^+, S^-  \subset \{\decor{1}, \dots, \decor{g}\}$ be subsets such that $|S^+| + |S^-| = 2r$ and let $W = V_{S^+} \otimes \overbar{V}_{S^-}$. Let $Q \subset W$ be the subspace spanned by exterior products of elements of $\bigoplus_{\i=1}^\g V_\i \otimes \overbar{V}_\i$. By \Cref{lemma:decomp} and the assumption that $\agrp{A}$ is torsion free, the quotient $W^{\lambda_A} / Q$ is nonzero if and only if there exists a relation between the normalized Frobenius eigenvalues of the form $\prod_{\decor{i} \in T^+} \lambda_{\decor{i}} \prod_{\decor{j} \in T^-} \lambda_{\decor{j}}^{-1} = 1$ where $T^+ \subset S^+$ and $T^- \subset S^-$ are disjoint, have nonempty union, and $|T^+ \sqcup T^-|$ is even.
\end{proof}

\Cref{thm:main-tate-thm} follows by combining \Cref{lemma:exceptional-class,lemma:wpr-exceptional-lambdas}.

\section{The existence of \texorpdfstring{$q$}{q}-Weil numbers with a given weighted permutation representation}
\label{sec:reconstr-algo}
In this section we give \Cref{alg:wpr-to-weil}, which, given an
appropriate number field, a geometrically simple weighted permutation representation
$\rho = (w,G)$, and a weakly $p$-admissible filtration of $\rho$, computes a $q$-Weil number with associated weighted permutation
representation $\rho$. We prove in \Cref{thm:alg-termination}
that this algorithm is correct. Similar approaches are described in
\cite{lenstra-oort}, \cite[Section~4]{Zarhin94}, and \cite[Appendix~A]{Ancona-standardconjs} in special cases. By abuse of notation for an ideal $I \subset \OO_L$ we write $v_{\mathfrak{P}}(I)$ for the power of $\mathfrak{P}$ appearing in the prime factorization of $I$. 

Let $K/\QQ$ be a CM field of degree $2g$ and let $L$ be its Galois
closure. For a $\Gal(L/K)$-stable ideal $I \subset \OO_L$ we write
$e(I)$ for the smallest positive integer such that for all prime
ideals $\mfp \subset \OO_K$, we have that
$e_{\mfp} \mid e(I) \nu_{\mfp}(I)$, where $e_{\mfp}$ is the
ramification index of $\mfp$ in the extension $L/K$ and
$\nu_{\mfp}(I) \coloneqq v_{\mfP}(I)$ for any prime $\mfP$ of $L$
dividing $\mfp$. The following lemma is immediate.

\begin{lemma}
  \label{lemma:ideal-intersect}
  Let $I \subset \OO_L$ be an ideal fixed by $\Gal(L/K)$, and let $J = I^{e(I)} \cap \OO_K$. Then $J\OO_L = I^{e(I)}$.  
\end{lemma}

\begingroup
\begin{algorithm}[H]
  \caption{
    \textsf{wprToWeilNumber}: An algorithm which outputs
    a $q$-Weil number with a specified weighted permutation
    representation.
  }
  \label{alg:wpr-to-weil}
  \begin{flushleft}
    \textbf{Input:} \hfill
    \begin{enumerate}
    \item \label{i:input1}
      A geometrically simple weighted permutation representation $\rho = (w,G)$ with $G$
      acting transitively on $X_{2g}$.
    \item \label{i:input2}
      A CM field $K$ of degree $2g$ and a Galois group $G$ (as a
      subgroup of $W_{2g}$ up to conjugacy).
    \item \label{i:input3}
      An isomorphism $\Aut(L) \cong G$ where $L$ is the Galois closure of $K$.
    \item \label{i:input4}
      A prime $\dpr \subseteq \OO_L$ such that $D \subseteq \Stab(w)$ where $D$ is the decomposition group of $\dpr$. 
    \end{enumerate}
    \textbf{Output:} A $q$-Weil number $\pi$ whose corresponding
    weighted permutation representation is $\rho$.
  \end{flushleft}
  \begin{algorithmic}[1]
    \STATE \label{step:I} Set $I \subseteq \OO_L$ to be the ideal
    \begin{equation*}
      I = \bigg(\prod_{\sigma \in G}\sigma(\dpr)^{w(\sigma^{-1}(1))}\bigg)^c \subseteq \OO_L
    \end{equation*}
    where $c \in \QQ_{>0}$ is the smallest rational number for which $I$
    is integral.
    
    \STATE \label{step:J} Set $J = I^{e(I)} \cap \OO_K$. 

    \STATE Set $\ord(J)$ to be the order of (the class of) $J$ in
    $\operatorname{Cl}(\OO_K)$.

    \STATE \label{step:alpha} Let $\alpha \in \OO_K$ be chosen so that
    $J^{\ord(J)} = \alpha \OO_K$.

    \STATE \label{step:k} Set $k = 2 v_p(\Nm_{K/\QQ}(\alpha)) / v_p(\Nm_{K/\QQ}(p))$
    so that $\alpha\oalpha\OO_{K^+} = p^k\OO_{K^+}$.
    
    \STATE \label{step:u} Set $u = \alpha \oalpha / p^k$.
    
    \IF{$u$ is a square in $\OO_{K^+}$}
    \STATE Choose $v \in \OO_{K^+}$ so that $v^2 = u$ and set $\eta = 1$.
    \ELSE
    \STATE Set $v = u$ and $\eta = 2$.
    \ENDIF
    
    \STATE Set $\pi = \alpha^{\eta}/v$.
    
    \RETURN $\pi$.
  \end{algorithmic}
\end{algorithm}
\endgroup

\begin{remark}
  \label{rmk:better-implementation}
  In the code associated to this article we provide a reference implementation of \Cref{alg:wpr-to-weil}. Our code relies on computing the splitting field $L$ and is therefore only practical when $G$ is small. It would be interesting to avoid this by working entirely in $K$, or by computing an $\dpr$-adic approximation to the action of $G$ on $L_{\dpr}$ (note that $p$ may be ramified in $L$) and verifying the output \emph{a posteriori}.
\end{remark}

\begin{theorem}
  \label{thm:alg-termination}
  \Cref{alg:wpr-to-weil} is correct.
\end{theorem}
\begin{proof}
  First, suppose that \Cref{alg:wpr-to-weil} returns $\pi \in \OO_K$ for which $\QQ(\pi) = K$. Let $I \subset \OO_L$ and $J = I^{e(I)} \cap \OO_K \subset \OO_K$ be the ideals defined in Steps~\ref{step:I}~and~\ref{step:J} of \Cref{alg:wpr-to-weil}. As in Step~\ref{step:alpha} let $\alpha \in \OO_K$ be a generator for the principal ideal $J^{\ord(J)}$ where $\ord(J)$ is the order of the class of $J$ in $\operatorname{Cl}(\OO_K)$. Take $k = 2 v_p(\Nm_{K/\QQ}(\alpha)) / v_p(\Nm_{K/\QQ}(p))$ and $u = \alpha \oalpha / p^k \in \OO_{K^+}$, as defined in Steps~\ref{step:k}~and~\ref{step:u}.

  Recall that we write $v_{\dpr}$ for the valuation on $L$ associated to the prime $\dpr$. Let $\cR(\pi)$ denote the set of conjugates to $\pi$ in $L$ (i.e., the roots of the minimal polynomial of $\pi$ in $L$). It suffices to show the following:
  \begin{enumerate}[label=(\arabic*)]
  \item \label{i:pi-q-weil}
    $\pi$ is a $q$-Weil number,

  \item \label{i:val-correct}
    there exists $m \in \QQ$ such that for all $\tau \in G$ we have $m \cdot w(\tau( \decor{1} )) = v_{\dpr}(\tau( \decor{1} ))$, and
    
  \item \label{i:grp-act}
    the map $X_{2d} \rightarrow \cR(\pi) \colon \tau(\decor{1}) \mapsto \tau(\pi)$ is well defined, and is an isomorphism of $G$-sets.
  \end{enumerate}

  For \ref{i:pi-q-weil} note that if $u$ is a square in $\OO_{K^+}$, then $\pi = \frac{\alpha}{v}$ where $v^2 = u$ so 
  \begin{equation*}
    \pi \overline{\pi} = \frac{\alpha \oalpha}{v \overline{v}} = \frac{\alpha \oalpha}{v^2} = \frac{\alpha \oalpha}{u} = p^k.
  \end{equation*}
  In particular, the absolute value of $\pi$ in the complex numbers is $p^{k/2}$. Similarly, if $u$ is not a square in $\cO_{K^+}$ then $\pi = \frac{\alpha^2}{u}$ and
  \begin{equation*}
    \pi \overline{\pi} = \left( \frac{\alpha \oalpha}{u} \right)^2 = p^{2k}.
  \end{equation*}
  Thus the absolute value of $\pi$ in the complex numbers is $p^{k}$. In both cases $\pi$ is a $q$-Weil number (with $q = p^{k}$, respectively $q = p^{2k}$).

  To show \ref{i:val-correct} we compute the prime factorization of $\tau(\alpha) \OO_L$ for each $\tau \in G$. Note that by \Cref{lemma:ideal-intersect} we have $J\cO_L = I^{e(I)}$, and therefore for all $\tau \in G$ we have
  \begin{align*}
  \label{i:calc}
    \tau(\alpha) \OO_L &= \tau \left(I^{e(I) \ord(J)}\right) \\
                 &= \bigg( \prod_{\sigma \in G} \tau\sigma(\dpr)^{w(\sigma^{-1}(1))} \bigg)^{c e(I) \ord(J)} \\
                 &= \bigg( \prod_{\gamma \in G} \gamma(\dpr)^{w(\gamma^{-1}\tau(1))} \bigg)^{c e(I) \ord(J)}.
  \end{align*}
  Taking $\dpr$-adic valuations we have
  \begin{align*}
    v_{\dpr}(\tau(\alpha)) = c e(I) \ord(J) \sum_{\sigma \in D} w(\sigma^{-1} \tau(1)) ,
  \end{align*}
  where $D \subset G$ is the decomposition group of $\dpr$. But the ramification filtration of $\dpr$ is a weak $p$-admissible filtration of $\rho$ by assumption. Therefore $D \subset \Stab(w)$ and taking $m = c e(I) \ord(J) |D|$ proves \ref{i:val-correct}.

  To prove \ref{i:grp-act}, it suffices to show that $\Stab_G(\1) \subseteq \Stab_G(\pi)$. Let $\tau \in \Stab_G(\1)$ and note that $\tau$ stabilizes $I$, so $\tau(\alpha \cO_L) = \alpha \cO_L$. In particular, $\tau(\pi \cO_L) = \pi \cO_L$. Because $\rho$ is geometrically simple, the conjugates of $\pi$ generate $2g$ distinct principal ideals in $\cO_L$. Now, because $\deg(\pi) = 2g$, we have $\tau(\pi) = \pi$, as required.
  
  It remains to show that $\pi$ generates $K$. By assumption $\rho$ is geometrically simple and therefore $\Phi_{\rho}(x)\neq \Phi_{\rho}(y)$ for any distinct $x,y \in X_{2g}$. Hence the ideal $\pi\cO_L$ has $2g$ distinct $G$-conjugates, and in particular $\pi$ has degree $2g$, as required. 
\end{proof}

Let $f(x) = x^{2g} + a_{1} x^{2g-1} + \dots + a_{2g}$ be the minimal polynomial of a $q$-Weil number $\pi$. Define $\varepsilon_{\pi} = 1$ if $a_{2g} > 0$ and $\varepsilon_{\pi} = 2$ otherwise. Let
\[
  k_{\pi} = \lcm \bigg\{ \mathfrak{d} \bigg( \sum_{\j \in D\i} w(\j) \bigg) \bigg\}_{D\i \in D \backslash X_{2g}}
\]
where $\mathfrak{d}(x) = d$ denotes the denominator of a rational number $x = n/d$ where $\gcd(n,d) = 1$. 

\begin{proposition}
  \label{prop:dimension-of-return}
  Let $K$ be a CM field of degree $2g$ and Galois group $G$ (as a subgroup of $W_{2g}$ up to conjugacy), let $\rho = (w,G)$ be a geometrically simple weighted permutation representation with a weakly $p$-admissible filtration $G \supseteq D \supseteq G_0 \supseteq G_1$, and let $\pi$ be a $q$-Weil number output by \Cref{alg:wpr-to-weil}. Then $\pi$ is associated (by Honda--Tate theory) to the (isogeny class of an) abelian variety $A$ of dimension $g \cdot \lcm(\varepsilon_\pi, k_\pi)$.  In particular, the following are equivalent:
  \begin{enumerate}[label=(\roman*)]
  \item \label{i:p-adm}
    the sequence $G \supseteq D \supseteq G_0 \supseteq G_1$ is a strongly $p$-admissible filtration of $\rho$, and
  \item \label{i:same-dim}
    the (isogeny class of the) abelian variety associated to at least one of the $q$-Weil numbers $\pi$ or $\pi^2$ has dimension $g$. 
  \end{enumerate}
\end{proposition}

\begin{proof}
  Recall that we write $v$ for the $p$-adic valuation normalized so that $v(q) = 1$. By the Honda-Tate theorem, as stated in \cite[Theorem~4.2.12]{Poonen06}, the $q$-Weil number $\pi$ is associated to the isogeny class of an abelian variety $A$ of dimension $e g$, where $e$ is the smallest positive integer such that
  \begin{enumerate}
  \item
    $h(0)^e > 0$, and
  \item
    for each monic $\QQ_p$-irreducible factor $h(x)$ of $f(x)$, we have $v(h(0)^e) \in \ZZ$.  
  \end{enumerate}  

  The orbits of the action of $D$ on $X_{2g}$ are in bijective correspondence with the irreducible factors of $f(x)$ over $\QQ_p$. Let $h(x) \in \QQ_p(x)$ be a $\QQ_p$-irreducible factor of $f(x)$ and let $D \i$ be the corresponding orbit in $D \backslash X_{2g}$. Then we have
  \[
    v(h(0)) = \sum_{\j \in D\i} w(\j) 
  \]
  and the first statement of the proposition follows. 

  We now prove the equivalence of \ref{i:p-adm} and \ref{i:same-dim}. Note that $\pi^2$ and $\pi$ have isomorphic weighted permutation representations, up to scaling. The weighted permutation representation $\rho$ is strongly $p$-admissible if and only if $k_{\pi} = 1$. The claim follows from the fact that the trailing coefficient of the minimal polynomial of $\pi^2$ is positive since it is the square of the trailing coefficient of the minimal polynomial of $\pi$.
\end{proof}

\subsection{Proof of \Cref{thm:exceptions}}
We now prove \Cref{thm:exceptions} by combining \Cref{thm:main-tate-thm,thm:alg-termination} and \Cref{prop:dimension-of-return}.

We first show that if $A$ has an exceptional class, then the weighted permutation representation associated to $A$ is one of those listed in \cite{OurElectronic} (these cases are listed in \Cref{app:tables}  when $g \leq 5$). In those cases when $g$ is prime the claim follows from Tankeev's result~\cite[Theorem~1.2]{Tankeev1984} (\Cref{thm:others-coros}\ref{thm:tankeev-p} of this paper). For the remaining cases $g = 4, 6$ and for each weighting $w$ we enumerate ($\Stab(w)$-conjugacy classes of) subgroups of $W_{2g}$ and compute whether $\rho = (w,G)$ is:
\begin{enumerate}[label=(\arabic*)]
\item
  geometrically simple,
\item
  exceptional, and
\item
  has a filtration $G \supseteq D \supseteq G_0 \supseteq G_1$ for which:
  \begin{enumerate}[label=(\alph*)]
  \item \label{a:stabw}
    $D \subseteq \Stab(w)$, 
  \item \label{a:str-p-adm}
    $\sum_{\j \in D\i} w(\j) \in \ZZ$, and which
  \item \label{a:is-decomp}
    satisfies the conditions of \Cref{remark:conditions-on-D}.
  \end{enumerate}
\end{enumerate}
The claim follows by combining \Cref{thm:main-tate-thm} and \Cref{prop:AV-is-admissible}. 

Now, let $g \leq 5$. It remains to show that every exceptional geometrically weighted permutation representation $\rho = (w, G)$ listed in \cite{OurElectronic} is indeed strongly $p$-admissible, and occurs for some abelian variety of dimension $g$. Note that the latter implies the former. For every case $g < 5$, we provide in \Cref{sec:examples} either an explicit LMFDB link to an abelian variety with the required weighted permutation representation or a polynomial generating a CM number field, and a prime, such that the root is the $q$-Weil number associated to the isogeny class of an appropriate abelian variety. For $g = 5$, note that \textnormal{\cite{OurElectronic}} is empty, so the claim is vacuous.
\hfill\qed

\subsection{An interesting example}
\label{sec:an-explicit-example}
Let $K$ be the CM number field of degree $12$ whose LMFDB label~\cite{lmfdb} is \href{https://www.lmfdb.org/NumberField/12.0.1093889542148001.2}{\texttt{12.0.1093889542148001.2}}, and with minimal polynomial
\begin{equation*}
  f(T) = T^{12} - 6 T^{11} + 15 T^{10} - 18 T^{9} - 3 T^{8} + 24 T^{7} + 3 T^{6} + 6 T^{5} + 33 T^{4} + 152 T^{3} + 273 T^{2} + 198 T + 51.
\end{equation*}
The Galois group of $K$ is isomorphic to $A_4 \times C_2$ (acting as the transitive group \texttt{12T6}). Let $w$ be the weight function associated to the Newton polygon $[0,0,0,\tfrac{1}{3}, \tfrac{1}{3}, \tfrac{1}{3}, \tfrac{2}{3}, \tfrac{2}{3}, \tfrac{2}{3}, 1, 1, 1]$ and consider the subgroup
\begin{equation*}
   G = \langle (\decor{2} \decor{ 4} )(\decor{3} \bar{\decor{5}} )(\decor{5} \bar{\decor{3}} )(\bar{\decor{2}} \bar{\decor{4}} ), (\decor{1} \bar{\decor{3}} \bar{\decor{4}} )(\decor{2} \bar{\decor{6}} \decor{ 5} )(\decor{3} \decor{ 4} \bar{\decor{1}} )(\decor{6} \bar{\decor{5}} \bar{\decor{2}} ) \rangle \subset W_{12}.
\end{equation*}
In the notation defined in \cite[Section~2.5.1]{us} the group $G$ is denoted \texttt{12T6.12.t.a.177}. It is a simple computation to check that the weighted permutation representation $\rho = (w, G)$ is geometrically simple, has angle rank $\delta_\rho = 3$, and $G$ is isomorphic to the Galois group of $K$. 

By an explicit computation one checks that there exists a prime $\dpr$ above $2$ in the Galois closure $L/K$ for which the ramification filtration $G \supseteq D \supseteq G_0 \supseteq G_1$ is strongly $2$-admissible. In this case $2$ is unramified in $L$ and $D \cong \ZZ/3\ZZ$. Upon running \Cref{alg:wpr-to-weil} we find that the $8$-Weil polynomial
\begin{align*}
  P_8(T) = T^{12} - 12 T^{11} + 75 &T^{10} - 351 T^9 + 1392 T^8 - 4692 T^7 + 13912 T^6 - 37536 T^5 \\ &+ 89088 T^4 - 179712 T^3 + 307200 T^2 - 393216 T + 262144
\end{align*}
is a defining polynomial for $K$.

Indeed, there also exists a prime $\dpr$ above $3$ in the Galois closure $L/K$ for which the ramification filtration $G \supseteq D \supseteq G_0 \supseteq G_1$ is strongly $3$-admissible. In this case $D \cong G_0 \cong G_1 \cong \ZZ/3\ZZ$. Upon running \Cref{alg:wpr-to-weil} we find that the $3$-Weil polynomial
\begin{equation*}
  P_3(T) = T^{12} - 3 T^{11} + 14 T^9 - 21 T^8 - 27 T^7 + 120 T^6 - 81 T^5 - 189 T^4 + 378 T^3 - 729 T + 729
\end{equation*}
is a defining polynomial for $K$. See the file \href{https://github.com/SamFrengley/exceptional-tate-classes/blob/main/scripts/6-2-example.m}{\texttt{6-2-example.m}} in \cite{OurElectronic} for the computations verifying this example.

\begin{proof}[Proof of \Cref{coro:angle-corank-3}]
  By \Cref{prop:dimension-of-return} the $8$-Weil (respectively $3$-Weil) polynomial $P_8(T)$ (respectively $P_3(T)$) corresponds through the Honda--Tate theorem to an (isogeny class of an) abelian $6$-fold $A/\FF_8$ (respectively $A/\FF_3$) with endomorphism algebra $\End^0(A) \cong K$. In both cases the weighted permutation representation associated to $A$ is, by construction, $\rho$ (as defined above). In particular, by \Cref{lemma:geom-irred,lemma:ar-lemma}, $A$ is geometrically simple and has angle rank $3$. Finally a direct calculation shows that $\rho$ is not exceptional, and therefore $A$ carries no exceptional classes (by \Cref{thm:main-tate-thm}).
\end{proof}

\begin{remark}
  This example is particularly interesting because one can check (similarly to the proof of \Cref{thm:exceptions}) that $\rho$ is the \emph{only} strongly $p$-admissible weighted permutation representation associated to a geometrically simple abelian variety $A/\Fq$ such that:
  \begin{enumerate}
  \item
    $\End^0(A)$ is commutative,
  \item
    $g = 6$,
  \item
    the angle rank of $A$ is non-maximal (i.e., $\delta_A < 6$), and
  \item
    there are no exceptional classes in $H_\et^{2r}(A_{\Fqbar}, \QQ_\ell(r))$ for all $1 \leq r \leq g$.
  \end{enumerate}
  Moreover, by \Cref{thm:exceptions} every geometrically simple abelian variety of dimension $g < 6$ satisfying {\cond} and with no exceptional classes has angle rank $g$ or $g - 1$ (the latter only occurring when $g=3,5$).
\end{remark}

\subsection{Proof of \Cref{thm:CM-q-weil}}
\label{sec:proof-thm-CM}
\Cref{thm:CM-q-weil} follows from \Cref{prop:CM-q-weil-not-V4,prop:CM-q-weil-V4}. The first deals with the general case and the latter deals with the exceptional case when $K$ has Galois group $C_2 \times C_2$.

\begin{proposition}
  \label{prop:CM-q-weil-not-V4}
  Let $K/\QQ$ be a CM number field of degree $2g$ where $g$ is a prime number and suppose that the Galois group of $K$ is not isomorphic to $C_2 \times C_2$. Then there exists a prime power $q = p^k$ and a geometrically simple ordinary abelian variety $A/\Fq$ of dimension $g$ such that $K \cong \End^0(A)$.
\end{proposition}

\begin{proof}
  Let $L$ be the Galois closure of $K$. Since $K$ is a CM field the Galois group $\Gal(L/\QQ) \subseteq W_{2g}$ is a well-defined $W_{2g}$-conjugacy class. Choose a representative $G \subseteq W_{2g}$ for this conjugacy class, so that $G$ is not contained in $\operatorname{Sym}(\{\1, \dots \g\}) \times \langle \iota \rangle \subset W_{2g}$ where $\iota = (\1\bar\1)\dots(\g\bar\g)$ is the complex conjugation element (such a subgroup $G$ exists since $\Gal(L/\QQ) \not\cong C_2 \times C_2$). Let $\rho = (w, G)$ be the corresponding ordinary weighted permutation representation. It follows that the angle rank of $\rho$ satisfies $\delta_{\rho} > 1$. By \Cref{prop:refined-gen-tankeev}, the weighted permutation representation $\rho = (w,G)$ has angle rank $g$, and is therefore geometrically simple.

  By the Chebotarev density theorem there exists a prime number $p$ which is totally split in $L$. Fix a prime $\dpr \subset \OO_L$ above $p$ so that, in particular, the decomposition group at $\dpr$ is trivial and therefore the ramification filtration of $G$ with respect to $\dpr$ is a weakly $p$-admissible filtration of $\rho$. But, since $\rho$ is ordinary, $w(X_{2g}) = \{0,1\}$ and this filtration is automatically a strong $p$-admissible filtration. \Cref{thm:alg-termination} implies that \Cref{alg:wpr-to-weil} terminates on input, $\rho$, $K$, and $\dpr$. Combining \Cref{lemma:geom-irred} and \Cref{prop:dimension-of-return} and \Cref{prop:dimension-of-return}, we see that this output or its square corresponds to a geometrically simple ordinary abelian variety of dimension $g$, as required. 
\end{proof}

\begin{proposition}
  \label{prop:CM-q-weil-V4}
  Let $K$ be a CM quartic field with Galois group $\Gal(K/\QQ) \cong C_2 \times C_2$. Then there exists a prime power $q$ and an abelian surface $A/\Fq$ such that $K \cong \End^0(A)$ if and only if $K$ contains $\QQ(i)$ or $\QQ(\zeta_3)$ where $\zeta_3$ is a primitive $3^{\text{rd}}$-root of unity. Moreover, $A$ can be taken to be ordinary.
\end{proposition}

\begin{proof}
  Suppose that $K$ does not contain a root of unity except $\pm 1$, i.e., $K$ does not contain $\QQ(i)$ or $\QQ(\zeta_3)$. By direct calculation we see that for every weighting $w \colon X_{4} \to \QQ$ and transitive subgroup $G \subset W_{4}$ with $G \cong C_2 \times C_2$ the corresponding weighted permutation representation $\rho = (w, G)$ is not geometrically simple. Suppose in search of contradiction that $\pi \in K$ is a $q$-Weil number generating $K$. Since $\rho$ is not geometrically simple there exists a Galois conjugate $\pi' \in K$ of $\pi$ such that $\pi = \zeta\pi'$ for some root of unity $\zeta \in K$. But then $\pi = - \pi'$ and $\pi^2 = \alpha$ for some totally negative element $\alpha \in K^+$. Since $\pi$ is a $q$-Weil number, so too is $\alpha$. It follows that $\alpha = -q$ and $\pi = \sqrt{-q}$ which is a contradiction because $\sqrt{-q}$ does not generate a degree $4$ field.

  Now suppose that $K$ contains $\QQ(i)$ or $\QQ(\zeta_3)$. Since $K$ is a CM quartic field with Galois group $C_2 \times C_2$ it contains two distinct totally imaginary quadratic subfields $K_1, K_2 \subset K$. Without loss of generality we may assume that $K_1 = \QQ(i)$ or $\QQ(\zeta_3)$ and let $\zeta = i$, respectively $\zeta = \zeta_3$. Choose a prime number $p$ such that $p$ splits as a product $p\OO_{K_2} = \mfp \mfp'$. Note that $\mfp$ is principal if and only if $p$ is totally split in the Hilbert class field of $K_2$, so by the Chebotarev density theorem we may further assume that $\mfp$ is principal. Let $\alpha$ be a generator of $\mfp$, scaled so that $\alpha\oalpha = p$. Writing $\pi = \zeta \alpha$, we see that $\pi$ is an ordinary $p$-Weil number and a primitive element of $K$.
\end{proof}

\subsection{Proof of \Cref{thm:gen-CM-q-Weil}}
As in the statement of \Cref{thm:gen-CM-q-Weil} let $K$ be a CM field of dimension $2g$ where $g$ is a positive integer. Let $L$ be the Galois closure of $K$. Since $K$ is a CM field the Galois group $\Gal(L/\QQ) \subseteq W_{2g}$ is a well-defined $W_{2g}$-conjugacy class. Choose a representative $G \subseteq W_{2g}$ for this conjugacy class.

Similarly to the proof of \Cref{thm:CM-q-weil} we employ \Cref{thm:alg-termination} to produce a $q$-Weil number with our desired properties. However, in contrast to the proof of \Cref{thm:CM-q-weil} we ensure the $\rho = (w, G)$ is geometrically simple by adjusting the weight function $w$. In this case we lose control of the denominators in \Cref{prop:dimension-of-return} (thus losing the commutativity of the endomorphism algebra of the corresponding abelian variety).

\begin{lemma}
  \label{lemma:NP-lemma}
  Let $G \subseteq W_{2g}$ be a transitive subgroup, and let $w$ be the weight function associated to the Newton polygon with (distinct) slopes $s_1,\dots, s_m$ and horizontal lengths $l_1,\dots, l_m$. Then if $\gcd(l_1,\dots,l_m) = 1$, the weighted permutation $\rho = (w,G)$ is geometrically simple.
\end{lemma}
\begin{proof}
  Let
  \[
    X_{2g} = T_1 \sqcup \dots \sqcup T_n
  \]
  be the partition such that $\Phi_{\rho}(x) = \Phi_{\rho}(y)$ if and only if $x,y$ are contained in the same subset $T_{i}$. Note that each of the sets $T_{i}$ have the same cardinality and it suffices to show that $|T_i| = 1$. Note that for all $1 \leq i \leq n$ and $x,y \in T_i$ we have $w(x) = w(y)$. Therefore $|T_i|$ divides $l_{i}$ for all $1 \leq i \leq m$ and thus $|T_i| = 1$, as required.
\end{proof}

\begin{proof}[Proof of \Cref{thm:gen-CM-q-Weil}]
  Note that for all positive integers $g$ there exists a Newton polygon satisfying the conditions of \Cref{lemma:NP-lemma}. Indeed, we can take the ``K3 type'' Newton polygon from \Cref{thm:others-coros}\ref{thm:zarhin}. Setting $w$ to be the corresponding weight function, the weighted permutation representation $\rho = (w,G)$ is geometrically simple (by \Cref{lemma:NP-lemma}).

  By the Chebotarev density theorem there exists a prime number $p$ which is totally split in $L$, and we fix a prime $\dpr \subset \OO_L$ above $p$, so that, in particular, the decomposition group at $\dpr$ is trivial and $\rho$ is weakly $p$-admissible. \Cref{thm:alg-termination} implies that \Cref{alg:wpr-to-weil} terminates on input, $\rho$, $K$, and $\dpr$ returning a $q$-Weil number $\pi$. Let $A$ be an (isogeny class of an) abelian variety corresponding to $\pi$. By \Cref{prop:dimension-of-return}, if $w$ is chosen to correspond to the ``K3 type'' Newton polygon, then the dimension of $A$ is either $g$ or $2g$ (since $w(X_{2g}) = \{0,1/2,1\}$).

  The same argument as \Cref{lemma:geom-irred} shows that $A$ is geometrically simple (even when $\End^0(A)$ is not commutative).
\end{proof}

\begin{remark}
  When $G \not\cong C_2 \times C_2$, it would be interesting to strengthen \Cref{thm:gen-CM-q-Weil} to require $\End^0(A)$ to be commutative. By \Cref{thm:alg-termination} and \Cref{prop:dimension-of-return} it suffices to find a weight function $w$ and $W_{2g}$-conjugate of $G$ such that $\rho = (w, G)$ is strongly $p$-admissible and geometrically simple.
\end{remark}

\section{Relations to inverse Galois problems}
\label{sec:inverse-galo-probl}
In light of \Cref{alg:wpr-to-weil} and \Cref{thm:alg-termination} we now state a precise theorem (in the spirit of the inverse Galois problems posed in \cite[Conjecture~2.7]{DupuyKedlayaRoeVincent22} and \cite[Conjecture~7.1]{us}) describing the weighted permutation representations which may arise from a geometrically simple abelian variety.

\begin{theorem}
  \label{prop:ord-igp}
  Let $\rho = (w,G)$ be a geometrically simple weighted permutation representation. Let $p$ be a prime. The following are equivalent:
  \begin{enumerate}[label=(\roman*)]
  \item \label{it:av-exists}
    there exists a prime power $q = p^k$, a geometrically simple abelian variety $A/\Fq$ of dimension $g$, and an indexing of the Frobenius eigenvalues of $A$ such that $\rho$ is the weighted permutation representation corresponding to $A$, and 
  \item \label{it:CM-field-exists}
    there exists a CM field $K$ with Galois group $G$, a choice of conjugacy class $G \subseteq W_{2d}$, and a prime $\dpr$ in the Galois closure of $K$ such that $D \subseteq G$ is the decomposition group of $\dpr$, $\sum_{\decor{j} \in D\decor{i}}w(\decor{j}) \in \ZZ$, and $D \subseteq \Stab(w)$.
  \end{enumerate}
\end{theorem}

\begin{proof} 
  The claim \ref{it:av-exists} implies \ref{it:CM-field-exists} follows from \Cref{prop:AV-is-admissible}. The reverse implication follows from \Cref{alg:wpr-to-weil} and \Cref{thm:alg-termination}.
\end{proof}

Suppose we are given a weighted permutation representation $\rho = (w,G)$ and a strongly $p$-admissible filtration $G_1 \subseteq G_0 \subseteq D \subseteq G$ of $\rho$. Then there exists a surjection $\Gal(\overbar{\QQ}_p / \QQ_p) \rightarrow D$ such that the image of tame inertia is $G_1$ and the image of wild inertia is $G_0$. This surjection gives rise to a local  Galois extension $F/\QQ_p$ with Galois group $D$. If $\rho$ is geometrically simple and this local extension arises from a global CM-extension, then by \Cref{alg:wpr-to-weil} and \Cref{thm:alg-termination} it would give rise to a $q$-Weil number. The key question, therefore, is a local-to-global problem known as the Grunwald problem, which we phrase precisely here.

\begin{question}[The Grunwald problem]
\label{question:grunwald}
Let $G \subseteq S_n$ and let $S$ be a finite set of places of $\QQ$. For $v\in S$, let $F_v/\QQ_v$ be a Galois extension with embeddings $\Gal(F_v/\QQ_v) \rightarrow G$. Does there exist a Galois extension $L/\QQ$ with $\Gal(L/\QQ) \cong G$ such that $L_v \cong F_{v}^{\oplus \deg(L)/\deg(F_v)}$ for all $v \in S$?
\end{question}

In our case we let $S = \{p,\infty\}$. The question then becomes: when does there exist a Galois CM field $L/\QQ$ with $\Gal(L/\QQ)$ conjugate to $G$ and a prime $\mfP$ above $p$ such that the decomposition group at $\mfP$ is $D$?

The Grunwald problem very subtle. It was shown by Wang~\cite{wang} that when $p = 2$ local extensions may not lift to global ones. For example, class field theory shows that there are no Galois $C_8$-extensions $L/\QQ$ in which $2$ is inert, but there does exist a $C_8$-extension of $\QQ_2$. When $G$ is abelian, this failure of local lifting only occurs at $p = 2$ \cite{wang}. For general $G$ it remains unclear precisely when local extensions lift to global extensions (even when $p \neq 2$). However, when $G$ is supersolvable and $p$ does not divide $\lvert G\rvert$, Corollary 4.14 of \cite{harpaz-wittenberg:supersolvable} gives an affirmative answer to \Cref{question:grunwald}. 

The following conjecture is drawn from the literature on the Grunwald problem (e.g., \cite[Section~1.2]{dan} and \cite[pp. 1]{BN-tameGP}). 

\begin{conjecture}[Tame approximation]
\label{conj:tame-grunwald}
Let $G$ be a finite group and let $S$ be a finite set of places of $\QQ$ not dividing the cardinality of $G$. For each $v \in S$, let $F_{v}/\QQ_v$ be a Galois extension with an embedding $\Gal(F_{v}/\QQ_v) \hookrightarrow G$. Then there exists a Galois extension $L/\QQ$ with $\Gal(L/\QQ) \cong G$ and such that $L_v \cong F_{v}^{\oplus \deg(L)/\deg(F_v)}$ for all $v \in S$.
\end{conjecture}

Assuming the \Cref{conj:tame-grunwald}, we arrive at the following result.

\begin{proposition}
\label{coro:tame-grunwald}
Assume \Cref{conj:tame-grunwald}. Let $\rho = (w,G)$ be a geometrically simple weighted permutation representation. Assume there exists a strongly $p$-admissible filtration $G \supseteq D \supseteq G_0 \supseteq G_1$ of $\rho$ where $p$ is coprime to $\lvert G \rvert$. Then, there there exists a prime power $q = p^k$, a geometrically simple abelian variety $A/\Fq$ of dimension $g$, and an indexing of the Frobenius eigenvalues of $A$ such that $\rho$ is the weighted permutation representation corresponding to $A$.
\end{proposition}
\begin{proof}
This follows immediately from \Cref{prop:ord-igp} upon assuming \Cref{conj:tame-grunwald}.
\end{proof}

\appendix\section{Exceptional classes in small dimension}
\label{app:tables}

\subsection{Abelian varieties with angle rank \texorpdfstring{$\delta_A < g$}{d\_A < g}}
In \Crefrange{tab:np-first}{tab:np-last} we list every weighted permutation representation associated to a geometrically simple abelian variety of dimension $2 \leq g \leq 5$ satisfying {\cond} and having non-maximal angle rank $\delta_A < g$. In these tables we describe subgroups $G \subset W_{2g}$ by listing generators for the quotient $G/ \iota$ where $\iota = (\1\bar\1)\dots(\g\bar\g)$ is the complex conjugation element (which is necessarily contained in $G$). The tables also record whether each weighted permutation representation is exceptional, i.e., whether there exists an exceptional class in $H_\et^{2r}(A_{\Fqbar}, \QQ_\ell(r))^{G_{\Fq}}$ for some $1 \leq r \leq g$ (by \Cref{thm:main-tate-thm}). Those rows in which there are no exceptional classes are colored green. The first column of each row records the $\Stab(w)$-conjugacy class of $G$ in the labelling of \cite[Section~2.5.1]{us}, and the final column records an example from the LMFDB~\cite{DupuyKedlayaRoeVincent22,lmfdb} if one exists.

Note that not every Newton polygons occurs in \Crefrange{tab:np-first}{tab:np-last}. Indeed the Newton polygons which do not occur are those for which every geometrically simple abelian variety satisfying {\cond} has angle rank $\delta_A = g$ (this follows from the calculation in the proof of \Cref{thm:exceptions}). That is, we have proved the following lemma.

\begin{lemma}
  Let  $2 \leq g \leq 6$ be an integer and fix a Newton polygon (i.e., a weight function $w$) which is not supersingular. Then the following are equivalent:
  \begin{enumerate}[label=(\roman*)]
  \item
    every geometrically simple abelian variety $A$ with weight function $w$ satisfying {\cond} has maximal angle rank, and
  \item
    the Newton polygon is listed in \Cref{tab:max-ar}.
  \end{enumerate}
\end{lemma}

\begin{table}[H]
  \setlength{\arrayrulewidth}{0.3mm} 
  \setlength{\tabcolsep}{5pt}
  \renewcommand{\arraystretch}{1.2}
  \centering
  \begin{tabular}{|c|c|}
    \hline
    \rowcolor{headercolor}
    Dimension & Newton polygon  \\ \hline
    $2$ & $[0,0,1,1]$, $[0,\tfrac{1}{2},\tfrac{1}{2},1]$ \\
    \hline
    $3$ & $[0,0,0,1,1,1]$, $[0,\tfrac{1}{2},\tfrac{1}{2},\tfrac{1}{2},\tfrac{1}{2},1]$, $[\tfrac{1}{3},\tfrac{1}{3},\tfrac{1}{3},\tfrac{2}{3},\tfrac{2}{3},\tfrac{2}{3}]$ \\
    \hline
    $4$ & $[0,0,0,\tfrac{1}{2},\tfrac{1}{2},1,1,1]$, $[0,\tfrac{1}{2},\tfrac{1}{2},\tfrac{1}{2},\tfrac{1}{2},\tfrac{1}{2},\tfrac{1}{2},1]$, $[\tfrac{1}{3},\tfrac{1}{3},\tfrac{1}{3},\tfrac{1}{2},\tfrac{1}{2},\tfrac{2}{3},\tfrac{2}{3},\tfrac{2}{3}]$ \\
    \hline
    \multirow{3}{*}{$5$} & $[0,0,0,0,0,1,1,1,1,1]$, $[0,0,0,\tfrac{1}{2},\tfrac{1}{2},\tfrac{1}{2},\tfrac{1}{2},1,1,1]$, $[0,0,\tfrac{1}{3},\tfrac{1}{3},\tfrac{1}{3},\tfrac{2}{3},\tfrac{2}{3},\tfrac{2}{3},1,1]$, \\
                         & $[0,\tfrac{1}{4},\tfrac{1}{4},\tfrac{1}{4},\tfrac{1}{4},\tfrac{3}{4},\tfrac{3}{4},\tfrac{3}{4},\tfrac{3}{4},1]$, $[0,\tfrac{1}{3},\tfrac{1}{3},\tfrac{1}{3},\tfrac{1}{2},\tfrac{1}{2},\tfrac{2}{3},\tfrac{2}{3},\tfrac{2}{3},1]$, $[0,\tfrac{1}{2},\tfrac{1}{2},\tfrac{1}{2},\tfrac{1}{2},\tfrac{1}{2},\tfrac{1}{2},\tfrac{1}{2},\tfrac{1}{2},1]$, \\
                         & $[\tfrac{1}{5},\tfrac{1}{5},\tfrac{1}{5},\tfrac{1}{5},\tfrac{1}{5},\tfrac{4}{5},\tfrac{4}{5},\tfrac{4}{5},\tfrac{4}{5},\tfrac{4}{5}]$, $[\tfrac{1}{3},\tfrac{1}{3},\tfrac{1}{3},\tfrac{1}{2},\tfrac{1}{2},\tfrac{1}{2},\tfrac{1}{2},\tfrac{2}{3},\tfrac{2}{3},\tfrac{2}{3}]$, $[\tfrac{2}{5},\tfrac{2}{5},\tfrac{2}{5},\tfrac{2}{5},\tfrac{2}{5},\tfrac{3}{5},\tfrac{3}{5},\tfrac{3}{5},\tfrac{3}{5},\tfrac{3}{5}]$ \\
    \hline
    \multirow{3}{*}{$6$} & $[0,0,0,0,0,\tfrac{1}{2},\tfrac{1}{2},1,1,1,1,1]$, $[0,0,\tfrac{1}{3},\tfrac{1}{3},\tfrac{1}{3},\tfrac{1}{2},\tfrac{1}{2},\tfrac{2}{3},\tfrac{2}{3},\tfrac{2}{3},1,1]$,  \\
              & $[0,\tfrac{1}{5},\tfrac{1}{5},\tfrac{1}{5},\tfrac{1}{5},\tfrac{1}{5},\tfrac{4}{5},\tfrac{4}{5},\tfrac{4}{5},\tfrac{4}{5},\tfrac{4}{5},1]$, $[0,\tfrac{1}{2},\tfrac{1}{2},\tfrac{1}{2},\tfrac{1}{2},\tfrac{1}{2},\tfrac{1}{2},\tfrac{1}{2},\tfrac{1}{2},\tfrac{1}{2},\tfrac{1}{2},1]$,  \\
              &$[\tfrac{1}{5},\tfrac{1}{5},\tfrac{1}{5},\tfrac{1}{5},\tfrac{1}{5},\tfrac{1}{2},\tfrac{1}{2},\tfrac{4}{5},\tfrac{4}{5},\tfrac{4}{5},\tfrac{4}{5},\tfrac{4}{5}]$, $[\tfrac{2}{5},\tfrac{2}{5},\tfrac{2}{5},\tfrac{2}{5},\tfrac{2}{5},\tfrac{1}{2},\tfrac{1}{2},\tfrac{3}{5},\tfrac{3}{5},\tfrac{3}{5},\tfrac{3}{5},\tfrac{3}{5}]$  \\
    \hline
  \end{tabular}
  \caption{Newton polygons for which every geometrically simple abelian variety satisfying {\cond} has maximal angle rank.}
  \label{tab:max-ar}
\end{table}

\subsection{Further examples of $5$-folds}
\label{sec:examples}
Several of the weighted permutation representations for geometrically simple $5$-folds with non-maximal angle rank do not occur for any abelian variety in the LMFDB. In these cases we record examples which were found by combining \Cref{alg:wpr-to-weil} with the LMFDB number fields database~\cite{lmfdb}. In particular, we record the endomorphism algebras $K_A$ (which are CM fields of degree $2g$) via their LMFDB label. These further examples are listed in \Crefrange{tab:eg-1}{tab:eg-3}.

Note that we do not provide explicit examples when $g = 5$ and $G \subset W_{10}$ is conjugate to the transitive group \texttt{10T11} and isomorphic to $C_2 \times A_5$ as an abstract group. Nevertheless, it is theoretically possible to compute such an example using \Cref{alg:wpr-to-weil}, as we now explain.

The CM field with LMFDB label \nfield{10.0.147723329623203.1} has Galois group conjugate to \texttt{10T11} and in each case there exists a prime number for which $\rho$ is strongly $p$-admissible. Explicitly:
\begin{enumerate}
    \item For the weighted permutation representation $\rho$ in Tables~\ref{tab:5-1} and \ref{tab:5-2} corresponding to the group \texttt{10T11}, there exist a sequence of subgroups
    \[
        G \supseteq D \supseteq G_0 \supseteq G_1, 
    \]
    with $D$ cyclic and $G_0 = G_1 = \{1\}$, which forms a strongly $p$-admissible filtration for all but finitely many choices of $p$. Choose $p$ so that $D$ is the decomposition group of some prime above $p$ in the Galois closure of the field \nfield{10.0.147723329623203.1}; such a prime exists by the Chebotarev density theorem. 
    \item For the weighted permutation representation $\rho$ in Table~\ref{tab:np-last}, set $p = 3$. A computation shows that the decomposition data of a prime above $3$ in the Galois closure of the field \nfield{10.0.147723329623203.1} forms a strongly $p$-admissible filtration of $\rho$. 
\end{enumerate}

By \Cref{thm:alg-termination} running \Cref{alg:wpr-to-weil} will return the $q$-Weil number of an (isogeny class) of abelian variety of dimension $g$ with weighted permutation representation $\rho$. It may be possible to provide explicit examples by improving the efficiency our implementation of \Cref{alg:wpr-to-weil}; see \Cref{rmk:better-implementation} for a discussion of the inefficiencies of  \Cref{alg:wpr-to-weil}.

\vspace{2cm}

\begingroup
\footnotesize
\setlength{\abovecaptionskip}{2mm}

\begin{table}[H]
  \centering
  \begin{tabular}{|C{28mm}|C{70mm}|C{8mm}|C{17mm}|C{25mm}|}
  \hline 
  \rowcolor{headercolor} 
  $w_A$-conjugacy class & Generators for $G/\iota$ & $\delta_A$ & Exceptional & Example \\
  \hline 
  \rowcolor{goodcolor}\texttt{D6.6.t.a.2} & $(\decor{1} \bar{\decor{2}} \bar{\decor{3}} )(\decor{2} \decor{ 3} \bar{\decor{1}} ), (\decor{1} \bar{\decor{2}} )(\decor{2} \bar{\decor{1}} )$ & 2 & No & \avlink{3.2.ac_b_a} \\ 
  \hline 
  \end{tabular}
  \caption{Geometrically simple abelian 3-folds with Newton polygon $[0,0,\tfrac{1}{2},\tfrac{1}{2},1,1]$ and angle rank $\delta_A \neq g$.}
  \label{tab:np-first}
\end{table} 

\begin{table}[H]
  \centering
  \begin{tabular}{|C{28mm}|C{70mm}|C{8mm}|C{17mm}|C{25mm}|}
  \hline 
  \rowcolor{headercolor} 
  $w_A$-conjugacy class & Generators for $G/\iota$ & $\delta_A$ & Exceptional & Example \\
  \hline 
  \texttt{8T9.8.t.a.5} & $(\decor{1} \decor{ 2} )(\decor{3} \decor{ 4} )(\bar{\decor{1}} \bar{\decor{2}} )(\bar{\decor{3}} \bar{\decor{4}} ), (\decor{1} \bar{\decor{3}} \decor{ 4} \bar{\decor{2}} )(\decor{2} \bar{\decor{1}} \decor{ 3} \bar{\decor{4}} )$ & 3 & Yes & \avlink{4.3.ae_k_ay_bw} \\ 
  \hline 
  \texttt{8T13.8.t.a.2} & $(\decor{1} \decor{ 4} \bar{\decor{2}} \bar{\decor{1}} \bar{\decor{4}} \decor{ 2} )(\decor{3} \bar{\decor{3}} ), (\decor{1} \decor{ 2} )(\decor{3} \bar{\decor{4}} )(\decor{4} \bar{\decor{3}} )(\bar{\decor{1}} \bar{\decor{2}} )$ & 3 & Yes & \avlink{4.4.ab_af_a_bc} \\ 
  \hline 
  \texttt{8T24.8.t.a.1} & $(\decor{1} \bar{\decor{1}} )(\decor{2} \bar{\decor{3}} )(\decor{3} \bar{\decor{2}} )(\decor{4} \bar{\decor{4}} ), (\decor{1} \bar{\decor{3}} \bar{\decor{2}} \bar{\decor{4}} )(\decor{2} \decor{ 4} \bar{\decor{1}} \decor{ 3} )$ & 3 & Yes & \avlink{4.2.ac_b_c_ag} \\ 
  \hline 
  \end{tabular}
  \caption{Geometrically simple abelian 4-folds with Newton polygon $[0,0,\tfrac{1}{2},\tfrac{1}{2},\tfrac{1}{2},\tfrac{1}{2},1,1]$ and angle rank $\delta_A \neq g$.} 
\end{table}

\begin{table}[H]
  \centering
  \begin{tabular}{|C{28mm}|C{70mm}|C{8mm}|C{17mm}|C{25mm}|}
  \hline 
  \rowcolor{headercolor} 
  $w_A$-conjugacy class & Generators for $G/\iota$ & $\delta_A$ & Exceptional & Example \\
  \hline 
  \texttt{8T13.8.t.a.3} & $(\decor{1} \bar{\decor{4}} \bar{\decor{3}} \bar{\decor{1}} \decor{ 4} \decor{ 3} )(\decor{2} \bar{\decor{2}} ), (\decor{1} \bar{\decor{2}} \decor{ 4} )(\decor{2} \bar{\decor{4}} \bar{\decor{1}} )$ & 3 & Yes & \avlink{4.2.ac_b_ac_f} \\ 
  \hline 
  \texttt{8T24.8.t.a.4} & $(\decor{2} \bar{\decor{3}} )(\decor{3} \bar{\decor{2}} ), (\decor{1} \decor{ 3} \bar{\decor{4}} \bar{\decor{1}} \bar{\decor{3}} \decor{ 4} )(\decor{2} \bar{\decor{2}} )$ & 3 & Yes & \avlink{4.2.ad_c_g_ap} \\ 
  \hline 
  \end{tabular}
  \caption{Geometrically simple abelian 4-folds with Newton polygon $[0,0,0,0,1,1,1,1]$ and angle rank $\delta_A \neq g$.} 
\end{table}

\begin{table}[H]
  \centering
  \begin{tabular}{|C{28mm}|C{70mm}|C{8mm}|C{17mm}|C{25mm}|}
  \hline 
  \rowcolor{headercolor} 
  $w_A$-conjugacy class & Generators for $G/\iota$ & $\delta_A$ & Exceptional & Example \\
  \hline 
  \texttt{8T13.8.t.a.2} & $(\decor{1} \bar{\decor{2}} )(\decor{2} \bar{\decor{1}} )(\decor{3} \decor{ 4} )(\bar{\decor{3}} \bar{\decor{4}} ), (\decor{1} \bar{\decor{1}} )(\decor{2} \bar{\decor{3}} \decor{ 4} \bar{\decor{2}} \decor{ 3} \bar{\decor{4}} )$ & 3 & Yes & \avlink{4.3.ab_a_g_ag} \\ 
  \hline 
  \texttt{8T24.8.t.a.1} & $(\decor{1} \decor{ 2} \bar{\decor{4}} \bar{\decor{1}} \bar{\decor{2}} \decor{ 4} )(\decor{3} \bar{\decor{3}} ), (\decor{1} \bar{\decor{1}} )(\decor{2} \bar{\decor{3}} )(\decor{3} \bar{\decor{2}} )(\decor{4} \bar{\decor{4}} )$ & 3 & Yes & \avlink{4.2.ad_e_ag_k} \\ 
  \hline 
  \end{tabular}
  \caption{Geometrically simple abelian 4-folds with Newton polygon $[0,\tfrac{1}{3},\tfrac{1}{3},\tfrac{1}{3},\tfrac{2}{3},\tfrac{2}{3},\tfrac{2}{3},1]$ and angle rank $\delta_A \neq g$.} 
\end{table}

\begin{table}[H]
  \centering
  \begin{tabular}{|C{28mm}|C{70mm}|C{8mm}|C{17mm}|C{25mm}|}
  \hline 
  \rowcolor{headercolor} 
  $w_A$-conjugacy class & Generators for $G/\iota$ & $\delta_A$ & Exceptional & Example \\
  \hline 
  \texttt{8T13.8.t.a.3} & $(\decor{1} \decor{ 4} \bar{\decor{3}} )(\decor{3} \bar{\decor{1}} \bar{\decor{4}} ), (\decor{1} \bar{\decor{1}} )(\decor{2} \decor{ 4} \decor{ 3} \bar{\decor{2}} \bar{\decor{4}} \bar{\decor{3}} )$ & 3 & Yes & \avlink{4.4.ae_k_ay_ca} \\ 
  \hline 
  \texttt{8T24.8.t.a.4} & $(\decor{2} \bar{\decor{3}} \bar{\decor{4}} )(\decor{3} \decor{ 4} \bar{\decor{2}} ), (\decor{1} \bar{\decor{2}} \bar{\decor{4}} \decor{ 3} )(\decor{2} \decor{ 4} \bar{\decor{3}} \bar{\decor{1}} )$ & 3 & Yes & \avlink{4.2.ac_a_e_ag} \\ 
  \hline 
  \end{tabular}
  \caption{Geometrically simple abelian 4-folds with Newton polygon $[\tfrac{1}{4},\tfrac{1}{4},\tfrac{1}{4},\tfrac{1}{4},\tfrac{3}{4},\tfrac{3}{4},\tfrac{3}{4},\tfrac{3}{4}]$ and angle rank $\delta_A \neq g$.} 
\end{table} 

\begin{table}[H]
  \centering
  \begin{tabular}{|C{28mm}|C{70mm}|C{8mm}|C{17mm}|C{25mm}|}
  \hline 
  \rowcolor{headercolor} 
  $w_A$-conjugacy class & Generators for $G/\iota$ & $\delta_A$ & Exceptional & Example \\
  \hline
  \rowcolor{goodcolor}\texttt{D10.10.t.a.2} & $(\decor{2} \decor{ 3} )(\decor{4} \decor{ 5} )(\bar{\decor{2}} \bar{\decor{3}} )(\bar{\decor{4}} \bar{\decor{5}} ), (\decor{1} \decor{ 4} )(\decor{3} \bar{\decor{5}} )(\decor{5} \bar{\decor{3}} )(\bar{\decor{1}} \bar{\decor{4}} )$ & 4 & No & \Cref{tab:eg-1} \\ 
  \hline 
  \rowcolor{goodcolor}\texttt{10T5.10.t.a.5} & $(\decor{1} \bar{\decor{3}} \bar{\decor{2}} \decor{ 4} )(\decor{2} \bar{\decor{4}} \bar{\decor{1}} \decor{ 3} ), (\decor{1} \bar{\decor{1}} )(\decor{2} \bar{\decor{3}} )(\decor{3} \bar{\decor{2}} )(\decor{4} \bar{\decor{5}} )(\decor{5} \bar{\decor{4}} )$ & 4 & No & \Cref{tab:eg-1} \\ 
  \hline
  \rowcolor{goodcolor}\texttt{10T5.10.t.a.32} & $(\decor{1} \bar{\decor{5}} \bar{\decor{3}} \decor{ 4} )(\decor{3} \bar{\decor{4}} \bar{\decor{1}} \decor{ 5} ), (\decor{1} \bar{\decor{4}} )(\decor{2} \bar{\decor{3}} )(\decor{3} \bar{\decor{2}} )(\decor{4} \bar{\decor{1}} )(\decor{5} \bar{\decor{5}} )$ & 4 & No & \Cref{tab:eg-1} \\ 
  \hline
  \rowcolor{goodcolor}\texttt{10T11.10.t.a.4} & $(\decor{1} \decor{ 2} )(\decor{3} \decor{ 5} )(\bar{\decor{1}} \bar{\decor{2}} )(\bar{\decor{3}} \bar{\decor{5}} ), (\decor{1} \bar{\decor{5}} \bar{\decor{4}} )(\decor{4} \bar{\decor{1}} \decor{ 5} )$ & 4 & No & \avlink{} \\ 
  \hline 
  \rowcolor{goodcolor}\texttt{10T22.10.t.a.3} & $(\decor{1} \bar{\decor{3}} \decor{ 2} )(\decor{3} \bar{\decor{2}} \bar{\decor{1}} )(\decor{4} \decor{ 5} )(\bar{\decor{4}} \bar{\decor{5}} ), (\decor{1} \bar{\decor{5}} )(\decor{2} \bar{\decor{4}} )(\decor{4} \bar{\decor{2}} )(\decor{5} \bar{\decor{1}} )$ & 4 & No & \avlink{5.2.ae_g_ae_b_a} \\ 
  \hline 
  \end{tabular}
  \caption{Geometrically simple abelian 5-folds with Newton polygon $[0,0,0,0,\tfrac{1}{2},\tfrac{1}{2},1,1,1,1]$ and angle rank $\delta_A \neq g$.}
  \label{tab:5-1}
\end{table}

\begin{table}[H]
  \centering
  \begin{tabular}{|C{28mm}|C{70mm}|C{8mm}|C{17mm}|C{25mm}|}
  \hline 
  \rowcolor{headercolor} 
  $w_A$-conjugacy class & Generators for $G/\iota$ & $\delta_A$ & Exceptional & Example \\
  \hline
  \rowcolor{goodcolor}\texttt{D10.10.t.a.2} & $(\decor{1} \decor{ 5} )(\decor{2} \bar{\decor{4}} )(\decor{4} \bar{\decor{2}} )(\bar{\decor{1}} \bar{\decor{5}} ), (\decor{1} \decor{ 4} )(\decor{3} \bar{\decor{5}} )(\decor{5} \bar{\decor{3}} )(\bar{\decor{1}} \bar{\decor{4}} )$ & 4 & No & \Cref{tab:eg-2} \\ 
  \hline 
  \rowcolor{goodcolor}\texttt{10T5.10.t.a.1} & $(\decor{1} \decor{ 3} \bar{\decor{4}} \decor{ 5} )(\decor{2} \bar{\decor{2}} )(\decor{4} \bar{\decor{5}} \bar{\decor{1}} \bar{\decor{3}} ), (\decor{1} \bar{\decor{4}} \bar{\decor{3}} \bar{\decor{2}} \bar{\decor{5}} )(\decor{2} \decor{ 5} \bar{\decor{1}} \decor{ 4} \decor{ 3} )$ & 4 & No & \Cref{tab:eg-2} \\ 
  \hline
  \rowcolor{goodcolor}\texttt{10T11.10.t.a.2} & $(\decor{1} \bar{\decor{5}} \bar{\decor{2}} \bar{\decor{3}} \bar{\decor{4}} )(\decor{2} \decor{ 3} \decor{ 4} \bar{\decor{1}} \decor{ 5} ), (\decor{1} \bar{\decor{4}} \bar{\decor{3}} \bar{\decor{5}} \bar{\decor{2}} )(\decor{2} \bar{\decor{1}} \decor{ 4} \decor{ 3} \decor{ 5} )$ & 4 & No & \avlink{} \\ 
  \hline 
  \rowcolor{goodcolor}\texttt{10T22.10.t.a.2} & $(\decor{1} \decor{ 3} \decor{ 4} )(\decor{2} \bar{\decor{5}} )(\decor{5} \bar{\decor{2}} )(\bar{\decor{1}} \bar{\decor{3}} \bar{\decor{4}} ), (\decor{1} \decor{ 5} \bar{\decor{2}} \decor{ 4} \decor{ 3} )(\decor{2} \bar{\decor{4}} \bar{\decor{3}} \bar{\decor{1}} \bar{\decor{5}} )$ & 4 & No & \avlink{5.2.ac_b_a_a_a} \\ 
  \hline
  \end{tabular}
  \caption{Geometrically simple abelian 5-folds with Newton polygon $[0,0,\tfrac{1}{2},\tfrac{1}{2},\tfrac{1}{2},\tfrac{1}{2},\tfrac{1}{2},\tfrac{1}{2},1,1]$ and angle rank $\delta_A \neq g$.}
  \label{tab:5-2}
\end{table}

\begin{table}[H]
  \centering
  \begin{tabular}{|C{28mm}|C{70mm}|C{8mm}|C{17mm}|C{25mm}|}
  \hline 
  \rowcolor{headercolor} 
  $w_A$-conjugacy class & Generators for $G/\iota$ & $\delta_A$ & Exceptional & Example \\
  \hline 
  \rowcolor{goodcolor}\texttt{10T5.10.t.a.32} & $(\decor{1} \bar{\decor{5}} \bar{\decor{3}} \decor{ 4} )(\decor{3} \bar{\decor{4}} \bar{\decor{1}} \decor{ 5} ), (\decor{1} \decor{ 2} \decor{ 4} \decor{ 3} )(\decor{5} \bar{\decor{5}} )(\bar{\decor{1}} \bar{\decor{2}} \bar{\decor{4}} \bar{\decor{3}} )$ & 4 & No & \Cref{tab:eg-3} \\ 
  \hline
  \rowcolor{goodcolor}\texttt{10T11.10.t.a.4} & $(\decor{1} \bar{\decor{4}} \bar{\decor{5}} \decor{ 2} \bar{\decor{3}} )(\decor{3} \bar{\decor{1}} \decor{ 4} \decor{ 5} \bar{\decor{2}} ), (\decor{1} \bar{\decor{1}} )(\decor{2} \decor{ 4} )(\decor{3} \bar{\decor{5}} )(\decor{5} \bar{\decor{3}} )(\bar{\decor{2}} \bar{\decor{4}} )$ & 4 & No &  \\ 
  \hline 
  \rowcolor{goodcolor}\texttt{10T22.10.t.a.3} & $(\decor{1} \decor{ 2} \bar{\decor{5}} )(\decor{5} \bar{\decor{1}} \bar{\decor{2}} ), (\decor{1} \bar{\decor{3}} \bar{\decor{5}} )(\decor{2} \bar{\decor{4}} )(\decor{3} \decor{ 5} \bar{\decor{1}} )(\decor{4} \bar{\decor{2}} )$ & 4 & No & \avlink{5.2.a_ac_ae_c_m} \\ 
  \hline 
  \end{tabular}
  \caption{Geometrically simple abelian 5-folds with Newton polygon $[\tfrac{1}{4},\tfrac{1}{4},\tfrac{1}{4},\tfrac{1}{4},\tfrac{1}{2},\tfrac{1}{2},\tfrac{3}{4},\tfrac{3}{4},\tfrac{3}{4},\tfrac{3}{4}]$ and angle rank $\delta_A \neq g$.}
  \label{tab:np-last}
\end{table} 

\endgroup

\begingroup
\footnotesize

\begin{table}[H]
  \centering
  \begin{tabular}{|c|c|c|p{85mm}|}
    \hline
    \rowcolor{headercolor}
    $w_A$-conjugacy class  & $K_A$                          & $q$ & \multicolumn{1}{c|}{$P_A(T)$} \\
    \hline
    \texttt{D10.10.t.a.2} & \nfield{10.0.6283241669043.1} & $17^2$  &  $\begin{array}{l} T^{10} - 3 T^{9} + 223 T^{8} + 3689 T^{7} + 12955 T^{6} + 2383655 T^{5} \\ +3743995 T^{4} + 308108969 T^{3} + 5382677887 T^{2} \\ -20927272323 T + 2015993900449 \end{array}$      \\ \hline
    \texttt{10T5.10.t.a.5} & \nfield{10.0.22696305796096.1} & $47^2$  &  $\begin{array}{l} T^{10} + 50T^9 + 3518T^8 + 121740T^7 - 569290T^6 +\\ 103747612T^5 -  1257561610T^4 + 594052364940T^3 \\ +37921279527422T^2 + 1190564333088050T \\ +52599132235830049 \end{array}$      \\ \hline
    \texttt{10T5.10.t.a.32} & \nfield{10.0.22696305796096.1} & $3^2$  &  $\begin{array}{l} T^{10} + 9 T^{8} + 32 T^{7} + 110 T^{6} + 192 T^{5} + 990 T^{4} + 2592 T^{3} \\+ 6561 T^{2} + 59049 \end{array}$      \\ \hline
  \end{tabular}
  \caption{Examples of geometrically simple abelian $5$-folds with given weighted permutation representation, Newton polygon $[0,0,0,0,\tfrac{1}{2},\tfrac{1}{2},1,1,1,1]$, and non-maximal angle rank $\delta_A \neq g$.}
  \label{tab:eg-1}
\end{table}

\begin{table}[H]
  \centering
  \begin{tabular}{|c|c|c|p{85mm}|}
    \hline
    \rowcolor{headercolor}
    $w_A$-conjugacy class  & $K_A$                          & $q$ & \multicolumn{1}{c|}{$P_A(T)$} \\
    \hline
    \texttt{D10.10.t.a.2} & \nfield{10.0.6283241669043.1} & $17^2$ & $\begin{array}{l} T^{10} - 64 T^{9} + 1596 T^{8} - 8075 T^{7} - 542164 T^{6} + 15402255 T^{5} \\- 156685396 T^{4} - 674432075 T^{3} + 38523560124 T^{2} \\- 446448476224 T + 2015993900449\end{array}$ \\ \hline
    \texttt{10T5.10.t.a.1} & \nfield{10.0.22696305796096.1} & $47^2$  &  $\begin{array}{l} T^{10} - 134 T^{9} + 4100 T^{8} + 47658 T^{7} + 12317384 T^{6} \\ -1369010078 T^{5} + 27209101256 T^{4} + 232555837098 T^{3} \\+ 44194782848900 T^{2} - 3190712412675974 T \\+ 52599132235830049 \end{array}$      \\ \hline
  \end{tabular}
  \caption{Examples of geometrically simple abelian $5$-folds with given weighted permutation representation, Newton polygon $[0,0,\tfrac{1}{2},\tfrac{1}{2},\tfrac{1}{2},\tfrac{1}{2},\tfrac{1}{2},\tfrac{1}{2},1,1]$, and non-maximal angle rank $\delta_A \neq g$.}
  \label{tab:eg-2}
\end{table}

\begin{table}[H]
  \centering
  \begin{tabular}{|c|c|c|p{85mm}|}
    \hline
    \rowcolor{headercolor}
    $w_A$-conjugacy class  & $K_A$                          & $q$ & \multicolumn{1}{c|}{$P_A(T)$} \\
    \hline
    \texttt{10T5.10.t.a.32} & \nfield{10.0.22696305796096.1} & $3^4$  &  $\begin{array}{l} T^{10} + 81 T^{8} - 864 T^{7} + 8910 T^{6} - 46656 T^{5} + 721710 T^{4} \\- 5668704 T^{3} + 43046721 T^{2} + 3486784401 \end{array}$      \\ \hline
  \end{tabular}
  \caption{Examples of geometrically simple abelian $5$-folds with given weighted permutation representation, Newton polygon $[\tfrac{1}{4},\tfrac{1}{4},\tfrac{1}{4},\tfrac{1}{4},\tfrac{1}{2},\tfrac{1}{2},\tfrac{3}{4},\tfrac{3}{4},\tfrac{3}{4},\tfrac{3}{4}]$
, and non-maximal angle rank $\delta_A \neq g$.}
  \label{tab:eg-3}
\end{table}

\endgroup

\providecommand{\bysame}{\leavevmode\hbox to3em{\hrulefill}\thinspace}
\providecommand{\MR}{\relax\ifhmode\unskip\space\fi MR }
\providecommand{\MRhref}[2]{%
  \href{http://www.ams.org/mathscinet-getitem?mr=#1}{#2}
}
\providecommand{\href}[2]{#2}

\end{document}